\documentclass[a4paper]{amsart}
\usepackage{verbatim}
\usepackage[dvips]{color}
\usepackage{amssymb}
\usepackage[active]{srcltx}

\usepackage{latexsym}
\usepackage{amsmath}
\usepackage{float}
\usepackage{mathrsfs}
\usepackage{esint}
\usepackage{enumerate}

\newtheorem{theorem}{Theorem}[section]

\newtheorem{lemma}{Lemma}[section]
\newtheorem{corollary}{Corollary}[section]
\theoremstyle{definition}
\newtheorem{definition}{Definition}
\newtheorem{remark}{Remark}[section]
\newtheorem{example}{Example}

\floatstyle{ruled} \restylefloat{figure} \restylefloat{table}

\numberwithin{equation}{section}

\def\R{\mathbb R}

\def\A{\mathcal A}

\def\M{\mathcal M}

\def\X{\mathbf {X}}

\def\0{\boldsymbol 0}
\def\data{\X,\allowbreak C_D,\allowbreak C_P,\allowbreak \A_0,\allowbreak \A_1,\allowbreak p}

\newcommand{\CC}{Carnot-Carath\'eodory }
\newcommand{\supp}{\operatorname{supp}}

\newtoks\by
\newtoks\paper
\newtoks\book
\newtoks\jour
\newtoks\yr
\newtoks\pages
\newtoks\vol
\newtoks\publ

\def\name[#1, #2]{#1 #2}
\def\ota{{\hbox{\bf ???}}}
\def\cLear{\by=\ota\paper=\ota\book=\ota\jour=\ota\yr=\ota
\pages=\ota\vol=\ota\publ=\ota}
\def\endpaper{\the\by, \textit{\the\paper},
{\the\jour} \textbf{\the\vol} (\the\yr), \the\pages.\cLear}
\def\endbook{\the\by, \textit{\the\book},
\the\publ, \the\yr.\cLear}
\def\endpap{\the\by, \textit{\the\paper}, \the\jour.\cLear}
\def\endproc{\the\by, \textit{\the\paper}, \the\book, \the\publ,
\the\yr, \the\pages.\cLear}

\begin{document}

\title[Harnack estimates]{Harnack estimates for degenerate parabolic equations modeled on the subelliptic $p-$Laplacian}
\address{Benny Avelin\\Department of Mathematics, Uppsala University\\
S-751 06 Uppsala, Sweden}
\email{benny.avelin@math.uu.se}

\address{Luca Capogna\\Department of Mathematical Sciences, Worcester Polytechnic Institute\\Worcester, MA 01609
}
\email{lcapogna@wpi.edu}
\address{Giovanna Citti\\ Dipartimento di Matematica, Universita di Bologna\\
Bologna, Italy}
\email{citti@unibo.it}

\address{Kaj Nystr\"{o}m\\ Department of Mathematics, Uppsala University\\
S-751 06 Uppsala, Sweden}
\email{kaj.nystrom@math.uu.se}

\author{B. Avelin, L. Capogna, G. Citti and K. Nystr{\"o}m}

\maketitle
\begin{abstract}
	\noindent We establish a Harnack inequality for a class of quasi-linear PDE modeled on the prototype 
\begin{equation*}
	 \partial_tu= -\sum_{i=1}^{m}X_i^\ast ( |\X u|^{p-2} X_i u)\end{equation*}
where $p\ge 2$, $ \ \X = (X_1,\ldots, X_m)$ is a system of Lipschitz vector fields defined on a smooth manifold $\M$ endowed with a Borel measure $\mu$, and $X_i^*$ denotes the adjoint of $X_i$ with respect to $\mu$. Our estimates are derived assuming that (i) the control distance $d$ generated by $\X$ induces the same topology on $\M$; (ii) a doubling condition for the $\mu$-measure of $d-$metric balls and (iii) the validity of a Poincar\'e inequality involving $\X$ and $\mu$. Our results extend the recent work in \cite{DiBenedettoGianazzaVespri1}, \cite{K}, to a more general setting
including the model cases of (1) metrics generated by H\"ormander vector fields and Lebesgue measure; (2) Riemannian manifolds with non-negative Ricci curvature and Riemannian volume forms; and (3) metrics generated by non-smooth Baouendi-Grushin type vector fields and Lebesgue measure. In all cases the Harnack inequality continues to hold when the Lebesgue measure is substituted by any smooth volume form or by measures with densities corresponding to Muckenhoupt type weights.

%
\medskip

\noindent
2000 {\em Mathematics Subject Classification.}
\noindent

\medskip

\end{abstract}

 \setcounter{equation}{0} \setcounter{theorem}{0}

\section{Introduction and statement of main results}
\noindent
In their seminal works, Saloff-Coste \cite{SC} and Grigor'yan \cite{grig} established the equivalence between
Harnack inequalities for weak solutions to a class of subelliptic linear partial differential equations, with smooth coefficients, and two key metric-measure properties of the ambient space. The first property is the doubling inequality for the measure of balls, balls defined using a control metric naturally associated to the operator, and the second property is the validity of a Poincar\'e
inequality involving a notion of gradient naturally associated to the operator.
This point of view, independently developed in the work of Biroli and Mosco \cite{BM} and Sturm \cite{Sturm}, has been further studied by several authors, and has led to Harnack inequalities for more general classes of nonlinear parabolic PDE, see for instance \cite{KK}, \cite{KMMP}, \cite{MM}, \cite{MS} and \cite{CCR}. The ideas in \cite{SC} and \cite{grig}, are based on Moser's approach \cite{Moser2}. Although Moser's approach have been successfully used to prove Harnack's inequality for stationary solutions of equations of $p-$Laplace type, the extension of Moser's approach to the degenerate parabolic setting is not straightforward. Even in the Euclidean setting, the parabolic Harnack inequality for degenerate PDEs of $p-$Laplace type, with bounded and measurable coefficients, was only recently established by DiBenedetto, Gianazza and Vespri in \cite{DiBenedettoGianazzaVespri1} and by Kuusi in \cite{K}.

In this paper we add to this line of investigation by extending the recent works, \cite{DiBenedettoGianazzaVespri1} and \cite{K}, by establishing an intrinsic Harnack inequality for a class of quasi-linear differential equations tailored to the parabolic $p$-Laplacian in a general \CC setting. A {\it prototype} for the type of situations we consider in this paper is given by weak solutions to the degenerate parabolic quasi-linear PDE
 \begin{equation}\label{prototype}
 \partial_t  u (x,t) =- \sum_{i,j=1}^m \frac{1}{w(x)} X_i^* \bigg(w(x) |\X u(x,t) |^{p-2} a^{i,j}(x,t) X_j u(x,t) \bigg).
 \end{equation}
 Here $p\ge 2$, $\X=(X_1,...,X_m)$, with $X_i=\sum_{j=1}^n c_{ij}(x) \partial_{x_j}$,  is a system of smooth vector fields in $\R^n$, satisfying H\"ormander's finite rank hypothesis \cite{H}, $X^\ast_i=-X_i+\sum_{j=1}^n \partial_{x_j} c_{ij}(x)$ is the formal adjoint of $X_i$ (with respect to Lebsgue measure $d\mathcal L$), and $w(x)d\mathcal L$ is an admissible Borel measure (see Definition \ref{admissible} below).  The $m\times m$ matrix of (Lebesgue) measurable functions $a^{ij}(x,t)$ satisfies the usual coercivity hypothesis: {\it there exists $\lambda,\Lambda>0$ such that  $\Lambda |\xi|^2 \ge a^{ij}(x,t) \xi_i \xi_j \ge \lambda |\xi|^2$ for all $\xi\in \R^m$ and for a.e. $(x,t)$}. 

Although our results are new even in the setting of the example \eqref{prototype} in the  metric mesure spaces $(\R^n,d,w(x) d\mathcal L)$ with metrics $d$ associated to smooth H\"ormander vector fields, they actually encompass a broader setting. The motivation for pursuing  this larger degree of generality is two-fold:
\begin{enumerate}
	\item  We wish to identify to what extent, for a given measure metric space, the doubling and Poincar\'e inequalities are sufficient to guarantee parabolic Harnack type inequalities for solutions of problems involving operators of $p-$Laplace type (for instance for elements of De Giorgi classes or for solutions of gradient flows of the $p-$energy).
	\item Provide results that hold for non-smooth systems of vector fields generating
	control metrics which arise from applications, such as the Baouendi-Grushin system \cite{Baouendi}, \cite{Grushin}, or the vector fields appearing in the study of the Levi equation \cite{Citti}.
\end{enumerate}

We refer the reader to the exciting emerging literature on parabolic quasi-minimizers and parabolic De Giorgi classes in metric measure spaces, see \cite{KK}, \cite{MS}, \cite{KMMP}, \cite{MM}, for an alternative (and broader) point of view on the study of evolutionary problems in metric measure spaces equipped with a doubling measure, supporting a Poincar\'e inequality.

 \subsection{The ambient space geometry}
 We consider a smooth real manifold $\M$ endowed with a control distance $d(\cdot, \cdot): \M \times \M \to \R^{+}$ defined as the \CC control distance generated by a system of bounded, Lipschitz (when expressed in local coordinates) vector fields $\X = (X_1,\ldots, X_m)$ on $\M$, see \cite{NSW}, \cite{Bel} and \cite{GN1}. Following \cite{Bel} and \cite{GN2} our first standing hypothesis is that
 \begin{equation}\label{topology}
	\text{ the inclusion }	i: (\R^n, | \cdot |) \to (\R^n,d) \text{ is continuous.}
	\end{equation}
	This hypothesis guarantees that the topology generated on $\M$ by the metric $d$ coincides with the standard topology obtained as pull-back from the local charts of the standard topology in $\R^n$.
 We also request that $\X$ consists of	$\mu$-measurable vector fields on $\M$ where $\mu$ is a locally finite Borel measure on $\M$ which is absolutely continuous with respect the Lebesgue measure when represented in local charts.
	%
	%
		We let, for $x \in \M$ and $r > 0$, $B(x,r) = \{y \in \M: d(x,y) < r \}$ denote the corresponding open metric balls and we let $|B(x,r)|$ denote the $\mu$ measure of $B(x,r)$. In general, given a function $u$ and a ball $B=B(x,r)$ we will let $u_B$ denote the $\mu$-average of $u$ on the ball $B=B(x,r)$. In view of \eqref{topology} the closed metric ball $\bar B$ is a compact set.

%
We denote by $\sup$ and $\inf$ the \emph{essential supremum} and the \emph{essential infimum} defined with respect to $\mu$. Given a function $u$ on $\M$ we let $\supp u$ denote the support of $u$. If $p\ge 2$, and $u\in L^p_{loc}(\M,\mu)$ then the support is defined in terms of the support of a distribution.
Given $\Omega \subset \M$, open, and $1 \leq p \leq \infty$, we let $W_{\X}^{1,p}(\Omega) = \{ u \in L^p(\Omega,\mu): X_i u \in L^p (\Omega,\mu), i=1,...,m\}$ denote the {\it horizontal} Sobolev space, and we let $W_{\X,0}^{1,p}\subset W_{\X}^{1,p}$ be the closure\footnote{For a detailed study on the validity of ``$H=W$'' in general metric measure spaces, and in particular on the relation between the definitions used in this paper and the more commonly used definition based on the closure of the class of smooth functions with compact support see \cite{Friedrichs}, \cite{Kilpelainen}, \cite{GN2}, \cite{FSS1}, \cite{FSS2} and \cite{Serra}.} of the space of $W_{\X}^{1,p}$ functions with compact (distributional) support in the norm $\| u \|_{1,p}^p = \| u \|_p + \| \X u \|_p$ with respect to $\mu$. In the following we will omit $\mu$ in the notation for Lebesgue and Sobolev spaces.
Note that for $B(x,r)\subset \M$, the space
$$\{ \phi\cdot w \ |\ \phi\in C_{0}(B(x,r))\cap W_{\X}^{1,\infty}(B(x,r)), \text{ and }w\in W_{\X}^{1,p}(B(x,r))\},$$
where $C_{0}(B(x,r))$ is the set of continuous functions with support contained in $B(x,r)$, is a subset of $W_{\X,0}^{1,p}(B(x,r))$. Given
$t_1<t_2$, and $1 \leq p \leq \infty$, we let $\Omega_{t_1,t_2} \equiv \Omega \times (t_1,t_2)$ and we let $L^p(t_1,t_2;W_{\X}^{1,p}(\Omega))$, $t_1 < t_2$, denote the parabolic Sobolev space of real-valued functions defined on $\Omega_{t_1,t_2}$ such that for almost every $t$, $t_1 < t < t_2$, the function $x \to u(x,t)$ belongs to $W_{\X}^{1,p}(\Omega)$ and
\begin{equation*}
	\|u \|_{L^{p}(t_1,t_2;W_{\X}^{1,p}(\Omega))} =
	\left ( \int_{t_1}^{t_2} \int_{\Omega} (|u(x,t)|^p + |\X u(x,t)|^p) d\mu dt \right )^{1/p} < \infty.
\end{equation*}
The spaces $L^p(t_1,t_2;W_{\X,0}^{1,p}(\Omega))$ is defined analogously. We let $W^{1,p}(t_1,t_2;L^p(\Omega))$ consist of real-valued functions $\eta \in L^p(t_1,t_2;L^p(\Omega))$ such that the weak derivative $\partial_t\eta(x,t)$ exists and belongs to $L^p(t_1,t_2;L^p(\Omega))$. Consider the set of functions $\phi$, $\phi \in W^{1,p}(t_1,t_2;L^p(\Omega))$, such that the functions
\begin{equation*}
	t\to \int_{\Omega} |\phi(x,t)|^p d \mu(x) \mbox{ and } t\to \int_{\Omega} |\partial_t \phi(x,t)|^p d \mu(x),
\end{equation*}
 have compact support in $(t_1,t_2)$. We let $W_0^{1,p}(t_1,t_2; L^p(\Omega))$ denote the closure of this space under the norm in $W^{1,p}(t_1,t_2; L^p(\Omega))$.

Our second set of hypothesis is that we assume that $(\M,\mu,d)$ defines a so called {\it $p$-admissible structure} in the sense of \cite[Theorem 13.1]{HK}.

\begin{definition}\label{admissible} Assume hypothesis \eqref{topology} holds. Given $1 \leq p < \infty$, the triple $(\M,\mu,d)$ is said to define a $p$-admissible structure if for every compact subset $K$ of $\M$ there exist constants $C_D = C_D(\X, K), C_P = C_P(\X, K) > 0$, and $R = R(\X, K) > 0$, such that the following hold.
\begin{enumerate}
\item {\it Doubling property:}\begin{equation} \tag{D} \label{eq_D}
	 |B(x,2r)|\leq C_D |B(x,r)|\mbox{ whenever $x \in K$ and $0 < r < R$}.
\end{equation}
\item {\it Weak $(1,p)$-Poincar\'e inequality:}
\begin{equation}\tag{P} \label{eq_P}
	\fint_{B(x,r)} |u - u_B| d\mu \leq C_P \, r \left ( \fint_{B(x,2r)} |\X u|^p d\mu \right )^{1/p},
\end{equation}
whenever $x \in K$, $0 < r < R$,
$u\in W_{\X}^{1,p}(B(x,2r)).$

\end{enumerate}
\end{definition}
\subsection{Quasilinear degenerate parabolic PDE} From now on $(\M,\mu,d)$ will denote a $p$-admissible structure, for some $p\in [2,\infty)$, in the sense of Definition \ref{admissible}. Given a domain (i.e., an open, connected set) $\Omega\subset\M$, and $T>0$ we set $\Omega_T=\Omega\times (0,T)$. We will say that $\A$ is an {\it admissible symbol} (in $\Omega_T$) if the following holds:
\begin{enumerate}[(i)]
	\item  
	 $(x,t) \to \A(x,t,u,F)$ is measurable for every $(u,F) \in \R \times \R^m$,
	 \item $(u,F) \to \A(x,t,u,F)$ is continuous for almost every $(x,t) \in \Omega_T$,
	 \item the bounds
	\begin{equation}\label{admissiblesym}
		 \A(x,t,u,F) \cdot F \geq \A_0 |F|^p, \ |\A(x,t,u,F)| \leq \A_1 |F|^{p-1},
	\end{equation}
	hold for every $(u,F) \in \R \times \R^m$ and almost every $(x,t) \in \Omega_T$.
\end{enumerate}
 $\A_0$ and $\A_1$ are called the structural constants of $\A$. If $\A$ and $\tilde \A$ are both admissible symbols, with the same structural constants $\A_0$ and $\A_1$, then we say that the symbols are structurally similar.

Let $E$ be a domain in $\M \times \R$. We say that the function $u:E\to\R$ is a weak solution to
\begin{equation} \label{eq_theeq}
\partial_t u(x,t)=	L_{A,p} u \equiv -\sum_{i=1}^{m}X_i^\ast \A_i(x,t,u,\X u),
\end{equation}
 in $E$, where $X^\ast_i$ is the formal adjoint w.r.t. $d \mu$, if whenever $\Omega_{t_1,t_2} \Subset E$ for some domain $\Omega\subset\M$, $u \in L^p(t_1,t_2;W_{\X}^{1,p}(\Omega))$ and
	\begin{equation} \label{eq_sol}
		- \int_{t_1}^{t_2} \int_{\Omega} u \frac{\partial \eta}{\partial t} d\mu dt + \int_{t_1}^{t_2} \int_{\Omega} \A(x,t,u,\X u) \cdot \X \eta\ d\mu dt = 0,
	\end{equation}
	for every test function $$\eta \in W_0^{1,2}(t_1,t_2; L^2(\Omega)) \cap L^p (t_1,t_2; W_{\X,0}^{1,p}(\Omega)). $$ A function $u$ is a weak super-solution (sub-solution) to \eqref{eq_theeq} in $E$ if
whenever $\Omega_{t_1,t_2} \Subset E$ for some domain $\Omega\subset\M$, we have $u \in L^p(t_1,t_2;W^{1,p}(\Omega))$, and the left hand side of \eqref{eq_sol} is non-negative (non-positive) for all non-negative test functions $W_0^{1,2}(t_1,t_2; L^2(\Omega)) \cap L^p (t_1,t_2; W_{\X,0}^{1,p}(\Omega))$.

\subsection{Statement of main result}
The main result of the paper is the
 following Harnack inequality for weak solutions to \eqref{eq_theeq}.

%

\begin{theorem}\label{main-th} Let $(\M,\mu,d)$ be a $p$-admissible structure for some fixed $p\in [2,\infty)$. For a bounded open subset $\Omega\subset \M$, let $u$ be a non-negative, weak solution to \eqref{eq_theeq} in an open set containing the cylinder $\Omega\times [0,T_0]$ and assume that the structure conditions \eqref{admissiblesym} are satisfied.

There exist constants $C_1,C_2,C_3 \geq 1$, depending only on $\data$, such that for almost all $(x_0,t_0)\in \Omega\times [0,T_0]$, the following holds: If $u(x_0,t_0)>0$, and if $0<r\le R(\X, \bar \Omega)$ (from Definition \ref{admissible}) is sufficiently small so that
\begin{equation*}
	B(x_0,8r)\subset \Omega \quad \text{ and }\quad (t_0 - C_1 u(x_0,t_0)^{2-p}{r}^p,\
t_0 + C_1 u(x_0,t_0)^{2-p}{r}^p) \subset (0,T_0),
\end{equation*}
then
\begin{equation*}
	u(x_0, t_0) \le C_2\inf_{Q}u,
\end{equation*}
where
\begin{equation*}
	Q=B(x_0, {r}) \times \bigg(t_0 +\frac 1 2{C_3} { u(x_0,t_0)^{2-p}{r}^p},
t_0 + C_3 u(x_0,t_0)^{2-p}{r}^p\bigg).
\end{equation*}
Furthermore, the constants $C_1,C_2,C_3$ can be chosen independently of $p$ as $p\to 2$.
\end{theorem}

\begin{remark}
	The dependency of the constants in Theorem \ref{main-th} on the vector field $\X$ comes from the fact that the gradient bound on the cut-off functions established by Garofalo and Nhieu in \cite[Theorem 1.5]{GN2} (see Lemma \ref{lem_2.1}) depends on the Lipschitz constant of the vector fields.
\end{remark}

\begin{corollary} \label{main-cor} Let $(\M,\mu,d)$ be a $p$-admissible structure for some $p\ge 2$. Every weak solution of \eqref{eq_theeq} can be modified in a set of measure zero so that it is locally H\"older continuous with respect to the control distance.
\end{corollary}


\begin{remark}
	Note that in our set-up we have assumed that $\mu$ is absolutely continuous with respect to the Lebesgue measure when represented in local charts. In our arguments this hypothesis is necessary for the construction of suitable test functions. However, Theorem \ref{main-th}  remains true if this hypothesis is replaced by the assumption that the metric is differentiable in the direction of the vector fields, almost everywhere with respect to $\mu$, and that the differential is $\mu-$essentially bounded. In the latter case, although the formal adjoints $\X^\ast_i$, and hence \eqref{eq_theeq}, are not well defined, still, the notion of a weak solutions in \eqref{eq_sol} is well-defined and the theory applies.
\end{remark}

To put Theorem \ref{main-th} into perspective, and frame it within the context of the current literature, we note that Theorem \ref{main-th} contains, in terms of the structure conditions \eqref{admissiblesym}, the following examples and results as special cases.
\begin{example}
	In the case $\M=\mathbb R^n$, $d\mu$ equals the $n$-dimensional Lebesgue measure, $\X = (X_1,\ldots, X_m)=(\partial_{x_1},...,\partial_{x_n})$,
$p=2$, the result was established in the classical papers by Moser \cite{Moser2}, and by Aronson and Serrin \cite{AronsonSerrin}. The weighted version (with Muckenhoupt weights) 
  was  investigated by Chiarenza and Serapioni \cite{ChiSer}.
In the case $2\le p<\infty$, the corresponding Harnack inequality was proved by DiBenedetto, Gianazza and Vespri in \cite{DiBenedettoGianazzaVespri1}, see also \cite{DGVbook}, and by Kuusi using a different approach in \cite{K}.
\end{example}

\begin{example}
	In the case $(\M,\mu,d)$ is a $2$-admissible structure in the sense of Definition \ref{admissible} and $\A$ satisfies the structure conditions \eqref{admissiblesym} with $p=2$, the Harnack inequality was recently established by Rea and two of us in \cite{CCR}. In the broader context of parabolic De Giorgi classes (again $p = 2$) the Harnack inequality was proved  by Kinnunen, Marola, Miranda and Paronetto \cite{KMMP},  in a more general metric measure space setting.
\end{example}

%
In addition, Theorem \ref{main-th} also covers many new situations some of which we next exemplify.

\begin{example}\label{tre}  If $\M$ is a smooth manifold, $d\mu$ a smooth volume form, and $\X$ is a system of smooth vector fields satisfying H\"ormander's finite rank condition $rank(Lie\{\X\})(x) =n$ at every point $x\in \M$ (see \cite{H}), then
the Poincar\'e inequality is due to Jerison \cite{Jerison} and the doubling condition was established by Nagel, Stein and Wainger in \cite{NSW}.
The PDE \eqref{eq_theeq} is sub-elliptic and our results provide a (degenerate) parabolic analogue of the Harnack inequality established by Danielli, Garofalo and one of us in \cite{CDG}. Theorem \ref{main-th} also covers the case in which $d\mu$ can be expressed in local coordinates through a multiple of a smooth volume form times a Muckenhoupt $A_p$ weight with respect to the \CC metric generated by $\X$. In this weighted setting the Poincar\'e inequality is due to Lu \cite{Lu}. The stationary Harnack inequality for linear divergence form subelliptic equations  was first proved by Franchi, Lu and Wheeden \cite{FLW}. See also the interesting papers  \cite{FSS1},   \cite{FSS2}, and references therein.
\end{example}
\begin{example}
Our setting is also sufficiently broad to include non-smooth vector fields such as the Baouendi-Grushin frames, e.g., consider, for $\gamma\ge 1$ and $(x,y)\in \R^2$, the vector fields $X_1=\partial_x$ and $X_2=|x|^{\gamma}\partial_y$. Unless $\gamma$ is a positive even integer these vector fields fail to satisfy H\"ormander's finite rank hypothesis. However, the doubling inequality as well as the Poincar\'e inequality hold and have been used in the work of Franchi and Lanconelli \cite{FL} to establish Harnack inequalities for linear equations.
\end{example}

\begin{example} Consider a smooth manifold $\M$ endowed with a complete Riemannian metric $g$. Let $\mu$ denote the Riemann volume measure, and by $\X$ denote a $g-$orthonormal frame. If the Ricci curvature is bounded from below ($Ricci\ge -Kg$) then our result yields Harnack inequalities for non-negative weak solutions to \eqref{eq_theeq} in every compact subset of $(\M,g)$. In fact, in this setting the Poincar\'e inequality follows from Buser's inequality while the doubling condition is a consequence of the Bishop-Gromov comparison principle. If $K=0$, i.e. the Ricci tensor is non-negative, then these assumptions holds globally and so does the Harnack inequality. For more details, see \cite{chavel}, \cite{MSC} and \cite{HK}.\end{example}

\subsection{The proof of Theorem \ref{main-th} and further results} The main technical steps in the proof of Theorem \ref{main-th} are the following
weak Harnack inequalities.

\begin{theorem}\label{theorem1.2a}
Let $(\M,\mu,d)$ be a $p$-admissible structure for some given $p\ge 2$.  For a bounded open subset $\Omega\subset \M$, and $0<t_0<T_0$, consider a non-negative, weak super-solution $u$ of \eqref{eq_theeq} in an
open set containing the cylinder $\overline{B(x_0,8{r})}\times [t_0,t_0+T_0]$,
with $B(x_0,8{r})\subset \Omega$ and $0<r\le R(\X, \bar \Omega)$ (from Definition \ref{admissible}).

If the structure conditions \eqref{admissiblesym} are satisfied then
there exist constants $C_1,C_2 \geq 1$, depending only on $\data$, such that if, for a.e. $t_1\in (t_0, t_0+T_0)$, we set
\begin{equation*}
	T=\min\Bigg\{ T_0+t_0-t_1, \ C_1 {r}^p \Bigg( \fint_{B(x_0,{r})} u(x,t_1)d\mu\Bigg)^{2-p}\Bigg\},
\end{equation*}
and
\begin{equation*}
	Q=B(x_0,4{r}) \times ( t_1+{T}/{2}, t_1+T),
\end{equation*}
then
\begin{equation*}
\fint_{B(x_0,{r})} u(x,t_1)d\mu \le \Bigg(\frac{C_1 {r}^p}{T_0+t_0-t_1} \Bigg)^{\frac{1}{p-2}} +C_2 \inf_Q u.
\end{equation*}
Furthermore, the constants $C_1,C_2$ can be chosen independently of $p$ as $p\to 2$.
\end{theorem}

%
%
%
\begin{theorem}\label{theorem1.2b}
Let $(\M,\mu,d)$ be a $p$-admissible structure for some given $p\ge 2$. For a bounded open subset $\Omega\subset \M$, and $0<t_0<T_0$ consider a non-negative, weak sub-solution $u$ of \eqref{eq_theeq} in an open set containing the cylinder $\overline{B(x_0,8{r})}\times [t_0-T_0,t_0]$,
with $B(x_0,8{r})\subset \Omega$ and $0<r\le R(\X, \bar \Omega)$ (from Definition \ref{admissible}).

If the structure conditions \eqref{admissiblesym} are satisfied then
there exists a constant $C \geq 1$, depending only on $\data$, such that
	\begin{equation*}
		\sup_{Q} u \leq C \left ( \frac{{r}^p}{T_0} \right )^{\frac{1}{p-2}} + C \frac{T_0}{{r}^p} \left ( \sup_{t_0-T_0 < t < t_0} \fint_{B(x_0,{r})} u d\mu \right )^{p-1},
	\end{equation*}
	where $Q = B(x_0,{r}/2) \times (t_0 - T_0/2,t_0)$. Furthermore, the constant $C$ can be chosen independently of $p$ as $p\to 2$.
\end{theorem}

%
%

Our proofs of Theorem \ref{theorem1.2a} and Theorem \ref{theorem1.2b} are loosely based on the strategy developed in \cite{K}, but also rely on the extension of certain arguments introduced in \cite{DiBenedetto} and \cite{DGVbook}.
Among our contributions, we single out exactly which assumptions are needed on the underlying geometry for the the results to hold. In particular, we
 modify the existing Euclidean arguments so they can be used in our broader setting, where rescalings in the space variables are not allowed and where there is no underlying group structure. Since, for $p>2$, time and space scaling are related, this rigidity introduces a further layer of technical difficulties.

%


The key steps in this proof are as follows.\\

\noindent
{\bf Expansion of positivity.} The important result here is Lemma \ref{lem_expofpos}. Indeed, to formulate an enlightening consequence of this lemma, let $Q \equiv B(x_0,4r_0) \times (t_0, t_0 + T_0)$, $B(x_0,4r_0) \Subset \Omega$, $0 < 4r_0 < R$, and let $u$ be a non-negative weak super-solution to \eqref{eq_theeq} in an open set containing $\overline{Q}$. Suppose that $t_0$ is a Lebesgue instant (see Definition \ref{lebesgue-instant}) for $u$ and
\begin{equation} \label{eq1}
	\big |\big \{x \in B(x_0,r): u(x,t_0) > M \big \}\big | \geq \delta |B(x_0,r)|.
\end{equation}
	for some $0 < r < r_0$, $M > 0$ and $0 < \delta < 1$. Then the conclusion is that there exists a positive constant $C$, independent of
$u$, $r$, $x_0$, $M$, $t_0$, $T_0$, but depending on $\delta$ and other structural parameters, so that
\begin{equation} \label{eq2}
	\inf_{B(x_0,r)} u(x,t_0+C M^{2-p} r^p ) \geq M.
\end{equation}
In particular, by expansion of positivity we mean that if $u(x,t_0)$ is large, on a substantial part of the ball $B(x_0,r)$, then we can use this
to derive a pointwise bound from below at the future instance defined by $t_0+CM^{2-p} r^p$. The proof of the estimate first uses a Caccioppoli inequality, together with the annular decay property stated in Lemma \ref{lem_2.1}, to conclude, (see Lemma \ref{lem_positivity} for the general statement), that
\begin{equation} \label{eq1uu}
	|\{x \in B(x_0,r): u(x,t) > \delta M/8 \}| \geq \frac{\delta}8|B(x_0,r)|,
\end{equation}
for all Lebesgue instants $t$ for $u$ satisfying $t_0 < t < M^{2-p} \delta^{p/\hat \delta+1} r^p / C$ where $\hat\delta$ occurs in the statement of the
annular decay property. Using \eqref{eq1uu}, and a special change of variables $t\to \Lambda(t)=\tau$, which exactly cancels the decay of the super-solution, but preserves the property of the function being a non-negative weak super-solution, one is able to conclude, this is Lemma \ref{lem_combo}, that the new super-solution $v$ satisfies
\begin{equation}\label{gh1}
		|\{ x \in B(x_0,r) : v(x,\tau) > 1 \}| \geq \nu |B(x_0,3r)|,
	\end{equation}
	for almost every $\tau^\ast \equiv \Lambda(t_0) < t < \Lambda (\hat T)$. Using \eqref{gh1} one can then, again via Caccioppoli inequalities, prove that for a Lebesgue instant in the future, the set where the super-solution is small can be made arbitrarily small in measure. This is then used to start a De Giorgi type iteration to conclude that the set where the function is small is zero in measure and hence, subsequently, obtaining \eqref{eq2}.\\

\noindent
{\bf Hot and Cold alternatives.} Based on the result concerning the expansion of positivity the proof of Theorem \ref{theorem1.2a} reduces, after some additional preliminary steps, to the consideration of two alternatives, Hot and Cold. Roughly speaking, in the first alternative, Hot, there is a time slice such that the solution is large in the sense that there exists a Lebesgue instant $t_0^\ast$ for $u$ satisfying $0 < t_0^\ast < C {r}^p$, such that
	\begin{equation} \label{eq4}
		|\{x \in B(x_0,r): u(x,t_0^\ast) > 8k^{1+\sigma} \}| > 8k^{-\sigma} |B(x_0,r)|,
	\end{equation}
	holds for some $k > 8^{1/\sigma}$, see Lemma \ref{hot}. In the second alternative, Cold, we have that
\begin{equation} \label{eq4te}
		|\{x \in B(x_0,r): u(x,t) > 8k^{1+\sigma} \}| \leq 8k^{-\sigma} |B(x_0,r)|,
	\end{equation}
	holds for every $k > 8^{1/\sigma}$ and for almost all $t$, $0 < t < C {r}^p$, see Lemma \ref{lem_5.1}. In either situation the goal is to be able to start the expansion of positivity, in order to establish
the existence of an instant $t_0$ at which the super-solution satisfies \eqref{eq1}. We next briefly discusses the underlying arguments used in the alternatives, Hot and Cold.
	\begin{enumerate}
		\item {\bf Hot:} We use a clustering lemma, see Lemma \ref{lem_local_clust}, to first to obtain, using \eqref{eq4}, a small ball in the time-slice $t_0^\ast$, in which $u$ satisfies \eqref{eq1} with $M=4k$ and $\delta = 1/2$. Our proof here is different to the proof developed in \cite{K} which is based on a covering type lemma together with a delicate analysis. The usage of Lemma \ref{lem_local_clust}, as an alternative to a covering type lemma, was first mentioned in \cite{DGVbook}.
		\item {\bf Cold:} As it turns out, \eqref{eq4te} implies that the Sobolev norm of a super-solution is small and that its average in space does not change to much in time. Thus the function has large average for each time-slice in some parabolic cylinder. Together with the small Sobolev norm one is then able to obtain that there is a time-slice inside this parabolic cylinder which satisfies \eqref{eq1}.
	\end{enumerate}

\section{Basic estimates}

\noindent
Throughout this section we will assume that $(\M,\mu,d)$ is a $p$-admissible structure for some $p\ge 1$, in the sense of Definition \ref{admissible}. We will also assume that $\Omega$ is a bounded open set in $\M$ and set $K=\bar\Omega$. The constants $C_D, C_P,$ and $R$ in Definition \ref{admissible} will all depend on $K$. Unless otherwise stated we let $C\geq 1$ denote a constant depending only on $C_D,C_P,p$, not necessarily the same at each occurrence.

\begin{lemma} \label{lem_2.1}
If $x \in K$ and $0 < s < r < R$, then the following holds.
	\begin{enumerate}
		\item There exists a constant $N = N(C_D) > 0$, called homogeneous dimension of $K$ with respect to $(\X, d, \mu)$, such that $|B(x,r)| \leq C_D \tau^{-N} |B(x,\tau r)|$, for all $0 < \tau \leq 1$.
		\item There exists a continuous function $\phi \in C_{0}(B(x,r))\cap W_{\X}^{1,\infty}(B(x,r))$ and a constant $C = C(\X,K) > 0$, such that $\phi = 1$ in $B(x,s)$ and $|\X \phi|\leq C/(r-s)$, $0 \leq \phi \leq 1$.
		\item Metric balls have the so called $\hat \delta-$annular decay property, i.e., there exists $\hat \delta = \hat \delta(C_D) \in (0,1]$, such that
		\begin{equation*}
			|B(x,r)\setminus B(x,(1-\epsilon)r)| \leq C \epsilon^{\hat \delta} |B(x,r)|,
		\end{equation*}
whenever $0 < \epsilon < 1$.
	\end{enumerate}
\end{lemma}

\begin{proof}
	Statement (1) follows from \eqref{eq_D} by a standard iteration argument. Statement (2) is proved in \cite[Theorem 1.5]{GN2}.
Statement 	(3) follows from \cite[Corollary 2.2]{Bu}, since we have a \CC space. Furthermore, $\hat \delta$ depends only on $C_D$.
\end{proof}

\begin{remark} 
	From now on, and in the subsequent lemmas, $N$ will play the role of the underlying dimension.
\end{remark}

\subsection{Parabolic Sobolev estimates}
We begin by stating a  result which is plays a fundamental role in  the development of  analysis on metric spaces. Note, in view of  \cite[Corollary 9.5]{HK}, that the metric balls  $B(x_0,r)$ are John domains. Hence, in view of properties \eqref{eq_D} and \eqref{eq_P}, and \cite[Theorem 9.7]{HK} one obtains the following  Sobolev-Poincar\'e inequality,  
\begin{lemma} \label{lem_sobolev}
	Let $B(x_0,r) \subset \Omega$, $0 < r < R$, $1 \leq p<\infty$. There exists a constant $C = C(C_D,C_P,p) \geq 1$ such that for every  $u \in W_{\X}^{1,p} (B(x_0,r))$,
	\begin{equation*}
		\left ( \fint_{B(x_0,r)} |u-u_B|^{\kappa p} d\mu \right )^{1/\kappa} \leq C r^p \fint_{B(x_0,r)} |\X u|^p d\mu,
	\end{equation*}
	where $u_B$ denotes the $\mu$ average of $u$ over $B(x_0,r)$,  and where $1 \leq \kappa \leq {N}/{(N-p)}$, if $1 \leq p < N$, and $1 \leq \kappa < \infty$, if $p \geq N$. Moreover, 
	\begin{equation*}
		\left ( \fint_{B(x_0,r)} |u|^{\kappa p} d\mu \right )^{1/\kappa} \leq C r^p \fint_{B(x_0,r)} |\X u|^p d\mu,
	\end{equation*}
	whenever $u \in W_{\X,0}^{1,p} (B(x_0,r))$.
\end{lemma}

We will also need the following  corollaries and reformulations of the Sobolev estimates. 

\begin{lemma} \label{lempoinmeas}
	Let $B(x_0,r) \subset \Omega$, $0 < r < R$, $1 \leq p<\infty$. Consider $u \in W_{\X}^{1,p}(B(x_0,r))$, let $A=\{x\in B(x_0,r): u=0\}$, and assume that $|A| > 0$. There exists a constant $C=C(C_D,C_P,p) \geq 1$ such that
	\begin{equation*}
		\left ( \fint_{B(x_0,r)} |u|^{\kappa p} d\mu \right )^{\frac{1}{\kappa p}} \leq C r \left ( \frac{|B(x_0,r)|}{|A|} \right )^{\frac{1}{\kappa p}} \left ( \fint_{B(x_0,r)} |\X u|^p d\mu
 \right )^{1/p},
	\end{equation*}
	whenever $1 \leq \kappa \leq {N}/{(N-p)}$, if $1 \leq p < N$, and $1 \leq \kappa < \infty$, if $p \geq N$.
\end{lemma}
\begin{proof} We let $u_B$ be the $\mu$-average of $u$ over the ball $B=B(x_0,r)$. Then by the definition of the set $A$ and $\kappa$ as in the statement of the lemma we first note that
	\begin{equation} \label{lem2.3eq1}
		|u_B| |A|^{\frac{1}{\kappa p}} \leq \left ( \int_{A} |u-u_B|^{\kappa p} d\mu \right )^{\frac{1}{\kappa p}}.
	\end{equation}
Using \eqref{lem2.3eq1} and the triangle inequality we see that
	\begin{eqnarray*}
		\left ( \fint_{B(x_0,r)} |u|^{\kappa p} d\mu \right )^{\frac{1}{\kappa p}} \leq 2 \left ( \frac{|B(x_0,r)|}{|A|} \right )^{\frac{1}{\kappa p}}\left ( \fint_{B(x_0,r)} |u-u_B|^{\kappa p} d\mu \right )^{\frac{1}{\kappa p}}.
	\end{eqnarray*}
The lemma now follows from H{\"o}lder's inequality and Lemma \ref{lem_sobolev}.
\end{proof}
%

\begin{lemma} \label{lem_parabolic:sobolev} Let $B(x_0,r) \subset \Omega$, $0 < r < R$, $1 \leq p<\infty$. Let $1 \leq \kappa < \infty$ and define
 $\kappa^\ast={N}/{(N-p)}$, if $1 \leq p < N$, and $\kappa^\ast=2$ if $p\geq N$. There exists a constant $C=C(C_D,C_P,p) \geq 1$ such that
	\begin{eqnarray*}
		\int_{t_1}^{t_2} \fint_{B(x_0,r)} |u|^{\kappa p} d\mu dt&\leq& C r^p \int_{t_1}^{t_2} \fint_{B(x_0,r)} |\X u|^{p} d\mu dt\notag\\
 &&\times\left ( \sup_{t_1 < t < t_2} \fint_{B(x_0,r)} |u|^{p \frac{(\kappa-1) \kappa^\ast}{\kappa^\ast-1}} d\mu \right )^{\frac{\kappa^\ast-1}{\kappa^\ast}},
	\end{eqnarray*}
for every $u \in L^p(t_1,t_2;W_{\X,0}^{1,p}(B(x_0,r)))$.
\end{lemma}
\begin{proof}
	Using H\"older's inequality and Lemma \ref{lem_sobolev} we have
		\begin{align*}
			\int_{t_1}^{t_2} &\fint_{B(x_0,r)} |u|^{\kappa p} d\mu dt \leq \int_{t_1}^{t_2} \fint_{B(x_0,r)} |u|^{p} |u|^{(\kappa-1)p} d\mu dt\notag \\
			&\leq \int_{t_1}^{t_2}\left ( \fint_{B(x_0,r)} |u|^{\kappa^\ast p} d\mu dt \right )^{\frac{1}{\kappa^\ast}} \left ( \sup_{t_1 < t < t_2} \fint_{B(x_0,r)} |u|^{p \frac{(\kappa-1) \kappa^\ast}{\kappa^\ast-1}} d\mu \right )^{\frac{\kappa^\ast-1}{\kappa^\ast}} \\
			&\leq C r^p\left ( \int_{t_1}^{t_2} \fint_{B(x_0,r)} |\X u|^{p} d\mu dt\right ) \left ( \sup_{t_1 < t < t_2} \fint_{B(x_0,r)} |u|^{p \frac{(\kappa-1) \kappa^\ast}{\kappa^\ast-1}} d\mu \right )^{\frac{\kappa^\ast-1}{\kappa^\ast}}.
		\end{align*}
\end{proof}
\begin{remark} Note that if $1 \leq p < N$, then $(2\kappa^\ast - 1)/{\kappa^\ast} ={(N+p)}/{N}$, and if $p \geq N$, then $(2\kappa^\ast - 1)/{\kappa^\ast}=3/2$.
\end{remark}
\begin{lemma} \label{cor:3.1}
	Let $x_0,r,p,\kappa^\ast$ be as in Lemma \ref{lem_parabolic:sobolev}. Let $u \in L^p(t_1,t_2;W_{\X,0}^{1,p}(B(x_0,r)))$ and let
$\{|u| > 0\} \equiv \{(x,t) \in B(x_0,r) \times (t_1,t_2): |u(x,t)| > 0\}$. There exists a constant $C = C(C_D,C_P,p)\geq 1$ such that
	\begin{eqnarray*}
		\int_{t_1}^{t_2} \fint_{B(x_0,r)} |u|^{p} d\mu dt
		&\leq& C \left ( \frac{|\{|u| > 0\}|}{|B(x_0,r)|} \right )^{\frac{\kappa^\ast - 1}{2\kappa^\ast-1}} r^{p \frac{\kappa^\ast}{2\kappa^\ast-1}}\\
&& \times \left ( \int_{t_1}^{t_2} \fint_{B(x_0,r)} |\X u|^{p} d\mu dt + \sup_{t_1 < t < t_2} \fint_{B(x_0,r)} |u|^{p} d\mu\right ).
	\end{eqnarray*}
\end{lemma}
\begin{proof} Firstly, using H\"older's inequality we see that
	\begin{eqnarray} \label{eqparsob2}
		\int_{t_1}^{t_2} \fint_{B(x_0,r)} |u|^{p} d\mu dt
		&\leq& C \left ( \frac{|\{|u| > 0\}|}{|B(x_0,r)|} \right )^{\frac{\kappa^\ast - 1}{2\kappa^\ast-1}}\notag\\
 &&\times \left ( \int_{t_1}^{t_2} \fint_{B(x_0,r)} |u|^{p\frac{2\kappa^\ast-1}{\kappa^\ast}} d\mu dt \right )^{\frac{\kappa^\ast}{2\kappa^\ast-1}}.
	\end{eqnarray}
	Secondly, using Lemma \ref{lem_parabolic:sobolev} we have that
	\begin{align} \label{eqparsob2+}
		\int_{t_1}^{t_2} &\fint_{B(x_0,r)} |u|^{p\frac{2\kappa^\ast-1}{\kappa^\ast}} d\mu dt \notag\\
		&\leq C r^{p } \int_{t_1}^{t_2} \fint_{B(x_0,r)} |\X u|^{p} d\mu dt \left ( \sup_{t_1 < t < t_2} \fint_{B(x_0,r)} |u|^{p} d\mu \right )^{\frac{\kappa^\ast-1}{\kappa^\ast}}.
	\end{align}
Finally, using \eqref{eqparsob2}, \eqref{eqparsob2+} and Young's inequality we can conclude that
	\begin{align*}
		\bigg ( \int_{t_1}^{t_2} &\fint_{B(x_0,r)} |u|^{p\frac{2\kappa^\ast-1}{\kappa^\ast}} d\mu dt \bigg )^{\frac{\kappa^\ast}{2\kappa^\ast-1}} \\
		&\leq C r^{p \frac{\kappa^\ast}{2\kappa^\ast-1}} \left ( \int_{t_1}^{t_2} \fint_{B(x_0,r)} |\X u|^{p} d\mu dt + \sup_{t_1 < t < t_2} \fint_{B(x_0,r)} |u|^{p} d\mu \right ),
	\end{align*}
and hence the proof is complete.
\end{proof}

\subsection{Parabolic De Giorgi Estimate}

\begin{lemma} \label{lem_degiorgi}
	Let $B(x_0,r) \subset \Omega$, $0 < r < R$, $1 \leq p<\infty$. Let $k$ and $l$ be any pair of real numbers such that $k < l$. There exists a constant $C=C(C_D,C_P,p) \geq 1$ such that
	\begin{align*}
		(l-k)|B(x_0,r) &\cap \{u > l\} |^{\frac{1}{\kappa p}}\notag \\
&\leq \frac{C r |B(x_0,r)|^{\frac{2}{\kappa p}}}{|B(x_0,r) \cap \{u < k\}|^{\frac{1}{\kappa p}} } \bigg ( \fint_{B(x_0,r)} \chi_{\{k < u < l\}}|\X u|^p d\mu \bigg )^{\frac{1}{p}},
	\end{align*}
whenever $u \in W_{\X}^{1,p} (B(x_0,r))$, and where $1 \leq \kappa \leq {N}/{(N-p)}$, if $1 \leq p < N$, and $1 \leq \kappa < \infty$, if $p \geq N$.
\end{lemma}

\begin{proof} The lemma can be proved by arguing as in \cite[Lemma 2.3]{KMMP} using Lemma \ref{lem_sobolev}.
\end{proof}

\subsection{Local Clustering in $W_{\X}^{1,p}$}
The lemma below is a metric space version of the so called ``Clustering Lemma'' from the work of DiBenedetto, Gianazza and Vespri \cite{DGVclust} (see also the earlier instance by DiBenedetto and Vespri \cite[Proposition A.1]{DBV}). It states that if a $W^{1,p}$ function is strictly positive in a large portion of the domain then the set where the function is positive clusters around one point. Our proof is a slight variant of the proof in \cite[Lemma 2.5]{KMMP}. In our version we have explicit estimates on the constants involved, something which will be used in the proof of the hot alternative (Lemma \ref{hot}).

\begin{lemma} \label{lem_local_clust}
	Let $B(x_0,r) \subset \Omega$, $0 < r < R$, $1 \leq p<\infty$. Consider $u \in W_{\X}^{1,p}(B(x_0,r))$ and assume that
	\begin{equation} \label{eqassumgrad}
		\left ( \fint_{B(x_0,r)} |\X u|^p d \mu \right )^{1/p} \leq \hat \gamma r^{-1} \quad \text{and} \quad |[u>1]\cap B(x_0,r)| \geq \alpha |B(x_0,r)|,
	\end{equation}
	for some $\hat \gamma > 0$ and $\alpha \in (0,1)$. There exists a constant $C = C(C_D,C_P,p) \geq 1$ such that for every $\delta, \lambda \in (0,1)$, there exist $\hat \epsilon= \hat \epsilon(C_D,C_P,p,\alpha,\delta,\hat \gamma,\lambda) \in (0,1)$, such that for any $0 < \epsilon \leq \hat \epsilon$ there exists $y \in B(x_0,r)$ satisfying
	\begin{equation*}
		\big |[u > \lambda] \cap B(y,\epsilon r)\big | > (1-\delta)
				|B(y,\epsilon r)|.
	\end{equation*}
	In particular,
	\begin{equation*}
		\hat \epsilon = \frac{1}{C} \min \left (\frac{\delta \alpha^{\frac{\kappa^\ast + 1}{\kappa^\ast p}}(1-\lambda)}{\hat \gamma}\, ,\, \alpha^{1/{\hat \delta}} \right ),
	\end{equation*}
	with $\kappa^\ast$ as in Lemma \ref{lem_parabolic:sobolev} and $\hat \delta$ as in Lemma \ref{lem_2.1}(3).
\end{lemma}

\begin{proof} Let $\epsilon_0>0$ be sufficiently small so that Lemma \ref{lem_2.1} yields
	\begin{equation*}
		|B(x_0,r)\setminus B(x_0,(1-\epsilon_0)r)| \leq \alpha/4 |B(x_0,r)|.
	\end{equation*}
For every $0<\epsilon<\epsilon_0$ consider a Vitali-Wiener type covering of $B(x_0,(1-\epsilon) r)$, see \cite{CW} and \cite[Theorem 3.2]{CW1}, to find a family of balls $\{B_i = B(w_i,r_i) \}$, $w_i \in B(x_0,r)$, $r_i \approx d(w_i,\partial B(x_0,r))$, such that
	\begin{enumerate}
		\item the balls $B_i\subset B(x_0,(1-\epsilon) r)$ are pairwise disjoint,
		\item $B(x_0,(1-\epsilon) r) \subset \cup_i 3B_i\subset B=B(x_0,r)$,
	\end{enumerate}
	and such that $C \epsilon r \geq r_i\ge \epsilon r$.
	Following \cite{DGVclust} we consider the sub-collection $B^+ = \{ B_j: |[u>1] \cap 3B_j| > \frac{\alpha}{k} |B_j|\}$ and denote by $B^-$ the rest of the balls. Here $k$ is a degree of freedom to be chosen. Since
\begin{equation*}
\sum_{B_j \in B^+} |3B_j| + \frac{\alpha}{k} \sum_{B_j \in B^-} |3B_j| \ge	\sum_{B_j \in B^+} |[u>1] \cap 3B_j| + \sum_{B_j \in B^-} |[u>1] \cap 3B_j| > \alpha/2 |B|,
\end{equation*}
then for $k$ sufficiently large depending only on the doubling constant $C_D$ one has
\begin{equation} \label{eqb+bdd}
	\sum_{B_j \in B^+} |3B_j| > \frac{\alpha}{4} |B|.
\end{equation}
Next, for every $B_j \in B^+$, setting
\begin{equation*}
	\left ( \fint_{B_j} |\X u|^p d\mu \right )^{1/p} \equiv D_j,
\end{equation*}
one can easily see that from Lemma \ref{lempoinmeas} and H\"older's inequality there exists a constant $C_2 = C_2(C_D,C_P,p) \geq 1$ such that
\begin{equation} \label{eq1-u+int}
	\fint_{B_j} (1-u)_+ d\mu \leq
	\frac{C_2 \epsilon r}{\alpha^{\frac{1}{\kappa^\ast p}}} \left ( \fint_{B_j} |\X u|^p d\mu	\right )^{1/p}
	=\frac{C_2 \epsilon r}{\alpha^{\frac{1}{\kappa^\ast p}}} D_j,
\end{equation}
Moreover
\begin{equation} \label{eq1-u+meas}
	(1-\lambda)\big | \big \{ u \leq \lambda \big \} \cap B_j \big| \leq \int_{B_j} (1-u)_+ d\mu.
\end{equation}
Combining \eqref{eq1-u+int} and \eqref{eq1-u+meas} yields
\begin{equation} \label{eqlambda}
	\big |\big \{u>\lambda \big \} \cap B_j \big| > \left ( 1 - \frac{C_2 \epsilon r}{(1-\lambda)\alpha^{\frac{1}{\kappa^\ast p}}} D_j \right )|B_j|.
\end{equation}
Next we show that for at least one $B_j\in B^+$ and for a constant $C_3 \geq 1$ depending only on $C_D$, one has
\begin{equation}\label{stima}
	D_j^p \le \frac{4 C_3\hat{\gamma}^p}{\alpha}r^{-p}.
\end{equation}
Inequality \eqref{stima}, together with \eqref{eqlambda} concludes the proof. To prove \eqref{stima} we note that by \eqref{eqb+bdd} and \eqref{eq_D} one has $\sum_{B_j\in B^+} |B_j|/|B|> \alpha/(4 C_3)$ for some $C_3=C_3(C_D) \geq 1$. On the other hand, from \eqref{eqassumgrad} it follows that
\begin{equation*}
	\sum_{B_j\in B^+} \frac{4 C_3}{\alpha} \frac{|B_j|}{|B|} D_j^p
	\le\frac{4 C_3}{\alpha} \left ( \fint_{B(x_0,r)} |\X u|^p d \mu \right )
	\le \frac{4 C_3}{\alpha} \hat\gamma^p r^{-p},
\end{equation*}
from which \eqref{stima} follows immediately. \end{proof}

%

\section{Estimates for sub/super-solutions}
\noindent
Throughout the rest of the paper we will assume that $(\M,\mu,d)$ is a $p$-admissible structure for some given $p\ge 2$, in the sense of Definition \ref{admissible}. We will assume that $\Omega$ is a bounded open set in $\M$ and set $K=\bar\Omega$. The constants $C_D, C_P,$ and $R$ in Definition \ref{admissible} will all depend on $K$. Unless otherwise stated we let $C\geq 1$ denote a constant depending only on $\data$, not necessarily the same at each occurrence.

%
\subsection{Caccioppoli estimate} Let $\zeta_h(s)$ be a standard mollifier with support in $(-h,h)$. Given $f : \M \times \R \to \R$, we define
\begin{equation*}
	f_h(x,t) = \int_{\R} f(x,s) \zeta_h(t-s) ds.
\end{equation*}
\begin{definition}\label{lebesgue-instant}
	Let $\Omega \subset \M$ be a domain, $u \in L^p(t_1,t_2;W_{\X}^{1,p}(\Omega))$, and consider $t_1 < t < t_2$. Then $t$ is called a Lebesgue instant for $u$ if
	\begin{equation*}
		\lim_{h \to 0} \int_{\Omega} |u_h (x,t)-u(x,t)|^2 d\mu = 0.
	\end{equation*}
\end{definition}

The following two lemmas can be proved in a standard fashion by proceeding along the lines of \cite{K} or \cite{K1}, hence we omit further details for the sake of brevity.

\begin{lemma} \label{lem_cacc}
	Let $\xi \in \R \setminus \{-1,0\}$, $\delta > 0$, and assume that $\A$ satisfies the structure conditions \eqref{admissiblesym}. If $u \geq \delta$ is a sub-solution (if $\xi> 0$) or a super-solution (if $\xi< 0$) to \eqref{eq_theeq} in $\Omega \times (\tau_1, \tau_2)$, then for any Lebesgue instants $t_1,t_2$ for $u$, with $\tau_1<t_1<t_2<\tau_2$, one has
	\begin{align*}
		\int_{t_1}^{t_2} \int_{\Omega} &|\X u|^{p} u^{\xi - 1} \phi^p d\mu dt + \frac{p}{\A_0 \xi} \int_{\Omega} \frac{u(x,t)^{1+\xi}}{1+\xi} \phi^p(x,t) d\mu \bigg |_{t=t_1}^{t_2} \\
		\notag &\leq
		\frac{p}{\A_0}\int_{t_1}^{t_2} \int_{\Omega} u^{1+\xi} \left ( \frac{1}{\xi(1+\xi)} \frac{\partial \phi^p}{\partial t} \right )_{+} d\mu dt \\
		&+ \left ( \frac{\A_1 p}{\A_0 |\xi|} \right )^p \int_{t_1}^{t_2} \int_{\Omega}u^{p - 1 + \xi} |\X \phi|^p d\mu dt,
	\end{align*}
	for all $\phi(x,t) = \psi(x) \zeta(t)$ with $\zeta \in C_0^{\infty}(\tau_1,\tau_2)$ and $\psi \in W_{\X,0}^{1,\infty}(\Omega)$.
	\end{lemma}

\begin{lemma} \label{cor_cacc}
	Let $\xi \in \R \setminus \{-1,0\}$, $\delta > 0$, and assume that $\A$ satisfies the structure conditions \eqref{admissiblesym}. If $u \geq \delta$ is a sub-solution (if $\xi> 0$) or a super-solution (if $\xi< 0$) to \eqref{eq_theeq} in $\Omega \times (\tau_1, \tau_2)$, then one has
	\begin{align*}
		\int_{\tau_1}^{\tau_2} \int_\Omega &|\X u|^p u^{\xi- 1} \phi^p d\mu dt + \frac{p}{\A_0 |\xi(1+\xi)|} \sup_{\tau_1 < t < \tau_2} \int_\Omega u^{1+\xi} \phi^p d\mu\\
		&\leq 2\left ( \frac{ \A_1 p}{\A_0 |\xi|} \right )^p \int_{\tau_1}^{\tau_2} \int_\Omega u^{p+\xi- 1} |\X \phi|^p d\mu dt \\
		&+ \frac{2p}{\A_0}\int_{\tau_1}^{\tau_2} \int_\Omega u^{1+\xi} \left ( \frac{1}{\xi(1+\xi)} \frac{\partial \phi^p}{\partial t} \right )_{+} d\mu dt,
	\end{align*}
	for all $\phi\in W_0^{1,\infty}(\tau_1,\tau_2; L^\infty(\Omega)) \cap L^\infty (\tau_1,\tau_2; W_{\X,0}^{1,\infty}(\Omega))$.\end{lemma}

\begin{remark} \label{lem3.2rem} Note that in Lemma \ref{lem_cacc} and Lemma \ref{cor_cacc} the constant $\delta$ is used only qualitatively. Furthermore, if $u$ is a non-negative super-solution to \eqref{eq_sol} w.r.t. the symbol $\A$, then $(u-k)_-$ is a bounded sub-solution of an equation w.r.t. to a symbol $\tilde \A$ which is structurally similar to $\A$, and if $\xi \geq 1$ one can apply apply both Lemma \ref{lem_cacc} and Lemma \ref{cor_cacc} with $u$ replaced by $(u-k)_-$ and with $\delta = 0$. In addition, we observe that
if $\psi \in W_{\X,0}^{1,\infty}(\Omega)$, and we set $\phi = \zeta(t) \psi$, where $\zeta \in C_0^\infty(\tau_1,\tau_2)$ and $\zeta = 1$ in $(t_1,t_2)$, then we can use $\psi$ directly, in place of $\phi$, in Lemma \ref{lem_cacc}.
\end{remark}

\subsection{Expansion of positivity for super-solutions}

The main technical tool used in the proof of the Harnack inequality is the following expansion of positivity lemma.

\begin{lemma}\label{lem_expofpos}
	Let $\A$ satisfy the structure conditions \eqref{admissiblesym}. Let $Q \equiv B(x_0,4r_0) \times (t_0, t_0 + T_0)$, $B(x_0,4r_0) \Subset \Omega$, $0 < 4r_0 < R$, let $u$ be a non-negative weak super-solution to \eqref{eq_theeq} in an open set containing $\overline{Q}$. Suppose that $t_0$ is a Lebesgue instant for $u$ and that
	\begin{equation*}
		\big |\big \{x \in B(x_0,r): u(x,t_0) > M \big \}\big | \geq \delta |B(x_0,r)|,
	\end{equation*}
	for some $0 < r < r_0$, $M > 0$ and $0 < \delta < 1$. There exist constants $\gamma = \gamma (\data,\allowbreak\delta)\geq 1$ and
	$\theta = \theta (\data,\allowbreak\delta) \geq 1$ such that
	\begin{equation*}
		\inf_{Q\cap Q'} u \geq M \left ( \frac{r}{r_0} \right )^{\theta},
	\end{equation*}
	where $Q' = B(x_0,2r_0) \times (t_0+\tilde T/2,t_0+\tilde T)$ and $\tilde T = \gamma \left ( M \left ( \frac{r}{r_0} \right )^{\theta} \right )^{2-p} r_0^p.$
\end{lemma}

\subsection{Auxiliary lemmas}

\begin{lemma} \label{lem_positivity} Assume that $\A$ satisfies the structure conditions \eqref{admissiblesym} and let $k > 0$ and $0 < \gamma < 1$. Let $u$ be a non-negative weak super-solution to \eqref{eq_theeq} in an open set containing $\overline{B(x_0,2r)} \times [0,k^{2-p} \gamma^{p/\hat \delta+1} r^p /C]$, with $\hat \delta$ as in Lemma \ref{lem_2.1}, and
$B(x_0,2r) \Subset \Omega$ with $0 < 2r < R$. If $t=0$ is a Lebesgue instant for $u$, and
	\begin{equation*}
		\big |\big \{ x \in B(x_0,r) : u(x,0) > k \big \}\big | \geq \gamma |B(x_0,r)|,
	\end{equation*}
	then
	\begin{equation*}
		\big |\big \{ x \in B(x_0,r) : u(x,t) > \gamma k / 8 \big \}\big | \geq \frac{\gamma}{8} |B(x_0,r)|,
	\end{equation*}
	holds for all Lebesgue instants $t$ for $u$ satisfying
	\begin{equation*}
		0 < t < k^{2-p} \gamma^{p/\hat \delta+1} r^p / C,
	\end{equation*}
	for a constant $C = C(\data) \geq 1$.
\end{lemma}
\begin{proof}
	Let $T_1 = k^{2-p} s^{p+1} r^p/C_1$ where $s$ and $C_1$ are degrees of freedom to be chosen. Using Lemma \ref{lem_2.1} we find a function $\phi \in W_{\X,0}^{1,\infty}(B(x_0,r))$ such that $\phi = 1$ in $B(x_0,(1-\epsilon)r)$, $0 \leq \phi \leq 1$ and $|\X \phi|\leq C/(\epsilon r)$ with $\epsilon$ a degree of freedom to be chosen. Using, as we may by Remark \ref{lem3.2rem}, $(u-k)_{-}$ and $\phi$ in Lemma \ref{lem_cacc} with $\xi = 1$, we obtain
	\begin{align}
		\int_{B(x_0,(1-\epsilon)r)} (u(x,\tau) - k)_{-}^2 d\mu \leq
		&\int_{B(x_0,r)} (u(x,0) - k)_{-}^2 d\mu \notag\\
		&+ C_2 \int_{0}^T \int_{B(x_0,r)} (u - k)_{-}^p |\X \phi|^p d\mu dt \notag\\
		\leq &k^2 (1-\gamma) |B(x,r)| + \frac{k^p C T_1}{(\epsilon r)^p} |B(x,r)| \notag\\
		\leq &k^2 \left ( 1-\gamma + \frac{C}{C_1}\frac{s^{p+1}}{\epsilon^p} \right ) |B(x,r)|, \label{lem3.4--}
	\end{align}
	for all Lebesgue instants $\tau$ for $u$ satisfying $0 < \tau < T_1$. Estimating the left hand side in \eqref{lem3.4--} we see that
	\begin{equation} \label{lem3.4-}
		\underset{B(x_0,(1-\epsilon)r)}\int (u(x,\tau) - k)_{-}^2 d\mu \geq (1-s)^2 k^2 |\{x \in B(x_0,(1-\epsilon)r) : u(x,\tau) \leq k s \}|.
	\end{equation}
	Using \eqref{lem3.4-} and \eqref{lem3.4--} we can conclude that
	\begin{align} \label{lem3.4+}
		|\{x \in B(x_0,(1-\epsilon)r) : u(x,\tau) \leq k s \}| \leq
		\frac{1-\gamma + \frac{C}{C_1}\frac{s^{p+1}}{\epsilon^p}} {(1-s)^2}|B(x,r)|.
	\end{align}
	Next, using the $\hat \delta-$annular decay property of Lemma \ref{lem_2.1}, and \eqref{lem3.4+}, we see that
	\begin{eqnarray*}
		|\{x \in B(x_0,r) : u(x,\tau) \leq k s \}| &\leq&
		|\{x \in B(x_0,(1-\epsilon)r) : u(x,\tau) \leq k s \}|\\
&& + C \epsilon^{\hat \delta} |B(x_0,r)| \\
		&\leq&
		\frac{1-\gamma + \frac{C}{C_1}\frac{s^{p+1}}{\epsilon^p} + C \epsilon^{\hat \delta}}{(1-s)^2}|B(x,r)|.
	\end{eqnarray*}
Given $\gamma$ we choose $\epsilon$ so that $C \epsilon^{\hat \delta} = \gamma/4$, and $s$ so that $s^{p+1}/\epsilon^p = \gamma/8^{p+1}$. Finally, we let
$C_1$ be determined by $C/C_1=1/4$. Using these parameters we can conclude that
	\begin{align*}
		|\{x \in B(x_0,r) : u(x,\tau) \leq k s \}| &\leq \frac{1-\gamma + \gamma/4 + \gamma/4}{\big (1-(\gamma \epsilon^p)^{1/(p+1)}/8\big )^2} |B(x_0,r)| \\
		&\leq \frac{1-\gamma/2}{1-\gamma/4} |B(x_0,r)|,
	\end{align*}
	for all Lebesgue instants $\tau$ for $u$ satisfying $0 < \tau < T_1$.
\end{proof}

\begin{lemma} \label{lem_3.3}
	Assume that $\A$ satisfies the structure conditions \eqref{admissiblesym} and let $0 < \delta < 1$. Let $u$ be a non-negative weak super-solution to \eqref{eq_theeq} in an open set containing $\overline{B(x_0,4r)} \times [0,\hat T]$, where $B(x_0,4r) \Subset \Omega$, and $0 < r < R$, and suppose that $0$ is a Lebesgue instant for u. There exist constants $C_1 = C_1(\data,\allowbreak\delta) \geq 1$ and $C_2 = C_2(\data,\allowbreak\delta) \geq 1$, such that for all $M > 0$, satisfying $\hat T > r^p / (C_1 M^{p-2})$ and
	\begin{equation*}
		|\{ x \in B(x_0,r): u(x,0) > M \}| \geq \delta |B(x_0,r)|,
	\end{equation*}
	there exists a Lebesgue instant $t^\ast$ for $u$ satisfying $0 < t^\ast < r^p / (C_1 M^{p-2})$, and a function
	\begin{equation*}
		w \in W_{\X,0}^{1,p}(B(x_0,2r)),
		\quad 0 \leq w \leq 1,
	\end{equation*}
	such that
	\begin{equation*}
		\big | \big \{ x \in B(x_0,r) : u(x,t^\ast) > \delta M / 8 \big \}\big | \geq \frac{\delta}{8} |B(x_0,r)|,
	\end{equation*}
	\begin{align*}
		w &= 1 \text{ a.e. in } \big \{x \in B(x_0,r): u(x,t^\ast) \geq \delta M / 8 \big \}, \\
		w &= 0 \text{ a.e. in } \big \{x \in B(x_0,2r): u(x,t^\ast) \leq \delta M / 16 \big \},
	\end{align*}
	and
	\begin{equation*}
		\fint_{B(x_0,2r)} |X w|^p d\mu \leq \frac{C_2}{r^p}.
	\end{equation*}
\end{lemma}

\begin{proof}
 Let $K=\frac{\delta}{8} M$. Since $u$ is a super-solution we see that $v=(u-k)_{-}$ is a non-negative sub-solution. Let $\psi \in W_{\X, 0}^{1,\infty}(B(x_0,2r))$ be the function in Lemma \ref{lem_2.1} such that
	\begin{equation*}
		0 \leq \psi \leq 1, \quad \psi = 1 \text{ in } B(x_0,r), \quad \text{and } |\X \psi| \leq \frac{C}{r}.
	\end{equation*}
Let $\zeta \in C^{\infty}(0,\hat T)$ be such that
	\begin{equation*}
		0 \leq \zeta \leq 1,\ \zeta(t)=1,\ \text{as}\ \hat T/2 \leq t \leq \hat T, \ \zeta(0) = 0,
	\end{equation*}
	and $\zeta_t \leq C / \hat T$. By Lemma \ref{cor_cacc} with $\xi = 1$, and Remark \ref{lem3.2rem} with test function $\psi \zeta$, we see that
	\begin{equation} \label{lem3.5+}
		\int_{\hat T/2}^{\hat T} \fint_{B(x_0,2r)} |\X (\psi v)|^p d\mu dt \leq C(k^p \hat T r^{-p} + k^2).
	\end{equation}
	Furthermore,
	\begin{equation*}
		\eta = \frac{2}{k} (k/2-v)_{+},
	\end{equation*}
	is a function such that $\eta = 0$ almost everywhere in $\{u \leq k/2\}$ and $\eta = 1$ almost everywhere in $\{u \geq k\}$. Moreover from \eqref{lem3.5+},
	\begin{eqnarray*}
		\int_{\hat T / 2}^{\hat T} \fint_{B(x_0,2r)} |\X (\psi \eta)|^p d\mu dt
		 &\leq&
		\frac{2^p}{k^p} \int_{\hat T / 2}^{\hat T} \fint_{B(x_0,2r)} |\X (\psi v)|^p d\mu dt \\
		&\leq& C(\hat T r^{-p} + k^{2-p}).
	\end{eqnarray*}
	Therefore, there exists a time $\hat T / 2 < t^{\ast} < \hat T$ such that
	\begin{equation} \label{lem3.5++}
		\fint_{B(x_0,2r)} |\X (\psi \eta(x,t^\ast))|^p d\mu \leq C \left ( r^{-p} + \frac{1}{\hat T k^{p-2}} \right ).
	\end{equation}
	Finally, we choose the function $w = \eta(\cdot,t^\ast) \psi$, which is in $w \in W_{\X,0}^{1,p}(B(x,2r))$ by \eqref{lem3.5++}. The largeness of the level-set at time $t^\ast$ follows from Lemma \ref{lem_positivity}, since we allow $C_1$ to depend on $\delta$.
\end{proof}

\begin{lemma} \label{lem_combo}
	Let $u,r, \hat T, M, \delta, C_1$ and $t^\ast$ be as in Lemma \ref{lem_3.3}. There exist constants $\kappa = \kappa (\data,\allowbreak \delta)\geq 1$, and $\nu = \nu(\data,\allowbreak \delta)\in (0,1)$, such that if we set \begin{equation*}
		\Lambda(t) = \frac{1}{\kappa^{p-1} (p-2) r^{-p}} \log(1+\kappa (p-2) \frac{1}{r^p} M^{p-2}t),
	\end{equation*}
and $\tau^\ast \equiv \Lambda(t^\ast)$, then $0 < \tau^\ast < \frac{r^p}{\kappa^{p-2} C_1}$ and
	\begin{equation*}
		v(x,t) = \frac{\kappa \exp(\kappa^{p-1} \frac{1}{r^p} t)}{M} u(x,\Lambda^{-1}(t)),
	\end{equation*}
	is a weak super-solution, in $B(x_0,4r) \times (0,\Lambda(\hat T))$, to an equation as in \eqref{eq_theeq} but with a new symbol, $\tilde \A$, which is structurally similar to $\A$. Furthermore,
	\begin{equation*}
		\big |\big \{ x \in B(x_0,3r) : v(x,t) > 1 \big \}\big | \geq \nu |B(x_0,3r)|,
	\end{equation*}
	for almost every $\tau^\ast < t < \Lambda (\hat T)$.
\end{lemma}
\begin{proof}
	Using Lemma \ref{lem_3.3} we see that Lemma \ref{lem_combo} follows along the same lines as the corresponding proof in \cite[Lemma 3.4, Lemma 3.5 and Corollary 3.6]{K}. We omit further details.
\end{proof}

\begin{lemma} \label{lem_smallness} Let $u,r, \hat T, M, \delta, C_1$ and $t^\ast$ be as in Lemma \ref{lem_3.3}, and let $v$, $\Lambda$, $\kappa$, $\tau^\ast$, be as in Lemma \ref{lem_combo}.
If $0 < \nu^\ast < 1$, there exists a constant $H = H(\data,\allowbreak \nu^\ast)\in\mathbb Z_+$ such that whenever
$$\tilde T \equiv \frac{r^p}{\kappa^{p-2}C_1} 2^{1+H(p-2)} < \Lambda(\hat T)/4,$$ then
	\begin{equation*}
		\big |\big \{ (x,t) \in B(x_0,3r)\times (\tilde T,4 \tilde T) : v(x,t) \leq 2^{-H} \big \}\big | \leq \nu^\ast |B(x_0,3r)\times (\tilde T, 4 \tilde T)|.
	\end{equation*}
	\end{lemma}

\begin{proof} In the following $H\in \mathbb Z_+$ is a degree of freedom to be chosen. Given $H$ we let $k_j = 2^{-j}$, for $j = 0,1,\ldots, H$, and $$\tilde T = 2 r^p \frac{k_H^{2-p}}{\kappa^{p-2}C_1} < \frac{\Lambda(\hat T)}{4}.$$
Let $\phi_1 \in W_{\X,0}^{1,\infty}(B(x_0,4r))$, be as in Lemma \ref{lem_2.1} such that $\phi_1 = 1$ in $B(x_0,3r)$. With $\tau^\ast$ as in Lemma \ref{lem_combo} we let $\phi_2 \in C^\infty(\tau^\ast,4 \tilde T)$ be such that $\phi_2 = 1$ for $t \in (\tilde T, 4 \tilde T)$ and such that $\phi_2(\tau^\ast) = 0$. Then $\phi = \phi_1 \phi_2$ is a test-function vanishing on the parabolic boundary of $B(x_0,4r) \times (\tau^\ast, 4 \tilde T)$, with $\tau^\ast$ as in Lemma \ref{lem_combo}, $\phi = 1$ in $B(x_0,3r) \times (\tilde T, 4\tilde T)$ and $|\X \phi|\leq C/r$, $|\phi_t|\leq C / \tilde T$. Note that $\tilde T > 2 \tau^\ast$.
	Using Remark \ref{lem3.2rem} we have that
	\begin{eqnarray} \label{eqpreI1I2+}
		\int_{\tau^\ast}^{4 \tilde T} \int_{B(x_0,3r)} |\X (v-k_j)_{-}|^p d\mu dt &\leq& C \int_{\tau^\ast}^{4 \tilde T} \int_{B(x_0,4r)} (v-k_j)_{-}^{2} \left (\frac{\partial \phi^p}{\partial t} \right)_{+} d\mu dt \notag\\
		&&+ C\int_{\tau^\ast}^{4 \tilde T} \int_{B(x_0,4r)} (v-k_j)_{-}^{p} |\X \phi|^p d\mu dt\notag\\
& \equiv& C (I_1 + I_2).
	\end{eqnarray}
	Using the definition of $k_j$ and $\phi$ we get
	\begin{equation}
		\label{eqI1I2+}
		I_1+I_2 \leq \frac{C k_j^p}{r^p}\big | B(x_0,4r) \times (\tilde T, 4\tilde T) \big|.
	\end{equation}
	Using \eqref{eqpreI1I2+}, \eqref{eqI1I2+}, Lemma \ref{lem_degiorgi}, Lemma \ref{lem_combo} and H\"older's inequality we obtain
	\begin{align*}
		\big |\big \{&(x,t) \in B(x_0,3r) \times (\tilde T, 4 \tilde T): v(x,t) \leq k_H \big \}\big | \\
		&\leq C |B(x_0,4r) \times (\tilde T, 4 \tilde T)|^{1/p} \bigg ( \int_{\tilde T}^{4 \tilde T} \int_{B(x_0,3r)} \chi_{\{ k_{j+1} < v < k_j \}} d\mu dt\bigg )^{\frac{p-1}{p}}.
	\end{align*}
	Taking the power $p/(p-1)$ on both sides, summing over $j=0,\ldots, H-1$, and using \eqref{eq_D}, we see that
	\begin{equation*}
		\big |\big \{(x,t) \in B(x_0,3r) \times (\tilde T, 4 \tilde T): v(x,t) \leq k_H \big \}\big | \leq \frac{C}{H^{(p-1)/p}} |B(x_0,3r) \times (\tilde T, 4 \tilde T)|.
	\end{equation*}
	We now let $H$ be the smallest integer larger than $\left ( {C}/{\nu^\ast} \right )^{\frac{p}{p-1}}$. This choice of $H$ completes the proof.
\end{proof}

\begin{lemma}\label{lem_preexpofpos} Let $u,r, \hat T, M, \delta, C_1$ and $t^\ast$ be as in Lemma \ref{lem_3.3}. There exist constants $\hat C \geq 1$ and $\mu^\ast \in (0,1)$ depending only on $\data,\allowbreak\delta$ such that if $\hat C M^{2-p} r^p < \hat T$, then
	\begin{equation*}
		\inf_{Q} u \geq M \mu^\ast,
	\end{equation*}
	where $Q = B(x_0,2r) \times (\hat C M^{2-p}r^p/2, \hat C M^{2-p}r^p)$.
\end{lemma}

\begin{proof}
	By Lemma \ref{lem_smallness} we find for every $\nu^\ast$ a constant $H$ such that
	\begin{equation} \label{eqk0}
		\big |\big \{ (x,t) \in B(x_0,3r)\times (\tilde T,4 \tilde T) : v(x,t) \leq 2^{-H} \big \}\big | \leq \nu^\ast |B(x_0,3r)\times (\tilde T, 4 \tilde T)|,
	\end{equation}
	where $\tilde T = \frac{r^p}{\kappa^{p-2}C_1} 2^{1+H(p-2)}$. We define
	\begin{equation*}
		k_j	= 2^{-H-1}(1+2^{-j}), \quad r_j = (2+2^{-j})r,\quad \hat T_j = 2 \tilde T(1-2^{-(j+1)p}),
	\end{equation*}
	for $j = 0,1,2,\ldots$, and construct the cylinders
	\begin{equation*}
		Q_j = B_j \times \Gamma_j = B(x_0,r_j) \times (\hat T_j, 4 \tilde T).
	\end{equation*}
	The sequence $k_j$ satisfies
	\begin{equation*}
		k_j - k_{j+1} = 2^{-H - j - 2} \quad \text{and} \quad 2^{-H-1} \leq k_j \leq 2^{-H}.
	\end{equation*}
Let $\phi_j = \psi_j \zeta_j$, where $\psi_j \in W_{\X,0}^{1,\infty} (B_j)$ is the function from Lemma \ref{lem_2.1} such that $\psi_j = 1$ in $B_{j+1}$, and $\zeta_j \in C_0^\infty(0,\Lambda(\hat T))$ which vanishes at $\hat T_j$, $\zeta_j = 1$ in $\Gamma_{j+1}$, $0 \leq \zeta \leq 1$, and
	\begin{equation*}
		|\X \phi_j| \leq C \frac{2^j}{r}, \quad (\phi_t)_+ \leq C \frac{2^{jp}}{\tilde T} \leq C 2^{jp} k_j^{p-2} \frac{\kappa^{p-2}C_1}{r^p} = C \frac{2^{jp} k_j^{p-2}}{r^p}.
	\end{equation*}
	Note that
	\begin{equation*}
		(v-k_j)_{-}^2 \geq \frac{(v-k_j)_{-}^p}{k_j^{p-2}}.
	\end{equation*}
	Then from Remark \ref{lem3.2rem}
	\begin{align*}
		\int_{Q_j} &|\X (v-k_j)_{-}|^p \phi_j^p d\mu dt + k_j^{2-p} \sup_{\Gamma_j} \int_{B_j} (v-k_j)_{-}^p \phi_j^p d\mu \\
		&\leq
		C \int_{Q_j} (v-k_j)_{-}^p |\X \phi_j|^p d\mu dt + C \int_{Q_j} (v-k_j)_{-}^2( (\phi_j^p)_t)_{+} d\mu dt \\
		&\leq
		C \frac{2^{jp}}{r^p} \bigg (\int_{Q_j} (v-k_j)_{-}^p d\mu dt
		+ k_j^{p-2} \int_{Q_j}(v-k_j)_{-}^2 d\mu dt \bigg ).
	\end{align*}
	We now change variables in time as we let $z = \frac{t \kappa^{p-2} C_1}{2^{1+H(p-2)}}$. Then
	\begin{align}\label{eq_00+}
		\int_{Q^z_j} &|\X (w-k_j)_{-}|^p \phi_j^p d\mu dz + C \sup_{\Gamma^z_j} \int_{B_j} (w-k_j)_{-}^p \phi^p d\mu \notag\\
		&\leq
		C \frac{2^{jp}}{r^p} \bigg (\int_{Q^z_j} (w-k_j)_{-}^p d\mu dz
		+ k_j^{p-2} \int_{Q^z_j}(w-k_j)_{-}^2 d\mu dz \bigg ),
	\end{align}
	where $Q^z_j = B_j \times \Gamma^z_j = B_j \times (2(1-2^{-(j+1)p})r^p, 4r^p)$, $w(x,z) = u(x,z \tilde T)$.
	Let
	\begin{equation*}
		A_j = \int_{Q^z_j} \chi_{\{ w < k_j \}} d\mu dz.
	\end{equation*}
	Notice that at the first level we have, see \eqref{eqk0}, that
	$A_0 \leq 9 \nu^\ast r^p |B(x_0,3r)|$. Using Lemma \ref{cor:3.1} and \eqref{eq_00+} we see that
	\begin{eqnarray} \label{lem3.8eq1}
		\int_{\Gamma^z_j} \fint_{B_j} (w-k_j)_{-}^p \psi_j^p d\mu dt &\leq&
		C r^{p \frac{\kappa^\ast}{2\kappa^\ast-1}}\bigg ( \frac{A_j}{|B_j|} \bigg )^{\frac{\kappa^\ast-1}{2 \kappa^\ast-1}} \frac{1}{|B_j|} E_1 \notag\\
		&\leq&
		 C r^{p \frac{\kappa^\ast}{2 \kappa^\ast - 1} - p}\bigg ( \frac{A_j}{|B_j|} \bigg )^{\frac{\kappa^\ast-1}{2 \kappa^\ast-1}} \frac{1}{|B_j|} E_2 \notag\\
		&\leq& C 2^{jp} r^{p \frac{2 \kappa^\ast - 1}{\kappa^\ast}-1} k_j^p \bigg ( \frac{A_j}{|B_j|} \bigg )^{1+\frac{\kappa^\ast-1}{2 \kappa^\ast-1}}.
	\end{eqnarray}
where 	\begin{eqnarray*}
	E_1& \equiv & \bigg ( \int_{Q^z_j} |\X (w-k_j)_{-}|^p \phi_j^p d\mu dz + \sup_{\Gamma^z_j} \int_{B_j} (w-k_j)_{-}^p \phi^p d\mu \bigg ),\\
E_2& \equiv & 2^{jp} \bigg ( \int_{Q^z_j} (w-k_j)_{-}^p d\mu dz
		+ k_j^{p-2} \int_{Q^z_j}(w-k_j)_{-}^2 d\mu dz \bigg ).
	\end{eqnarray*}
	Furthermore,
	\begin{eqnarray} \label{lem3.8eq2}
		\int_{\Gamma^z_j} \fint_{B_j} (w-k_j)_{-}^p \psi_j^p d\mu dt
		&\geq&
		\frac{1}{C_D}\int_{\Gamma^z_{j+1}} \fint_{B_{j+1}} (w-k_j)_{-}^p d\mu dt \notag\\
& \geq& \frac{1}{C_D} (k_j - k_{j+1})^p \frac{A_{j+1}}{|B_{j+1}|} \notag \\
		&\geq&
		\frac{k_j^p}{C 2^{j p}} \frac{A_{j+1}}{|B_{j+1}|}.
	\end{eqnarray}
Using \eqref{lem3.8eq1} and \eqref{lem3.8eq2} we get the inequality
	\begin{equation} \label{eq_00it}
		\frac{A_{j+1}}{|B_{j+1}|} \leq C 4^{j p} r^{p \frac{\kappa^\ast}{2 \kappa^\ast - 1}-p} \bigg ( \frac{A_j}{|B_j|} \bigg )^{1+\frac{\kappa^\ast-1}{2 \kappa^\ast-1}}
	\end{equation}
for $j=0,1,...$.

\noindent Defining $Y_j = A_{j}/(r^p |B_j|)$ we obtain \eqref{eq_00it} in a dimensionless form
	\begin{equation*}
		Y_{j+1} \leq C 4^{j p} Y_{j}^{1+\frac{\kappa^\ast-1}{2 \kappa^\ast-1}}.
	\end{equation*}
	By fast geometric convergence (Lemma 4.1, \cite{DiBenedetto}),
	\begin{equation} \label{eq_iter:bdd}
		A_j \to 0 \quad \text{if} \quad Y_0 \leq C^{-\frac{2 \kappa^\ast-1}{\kappa^\ast-1}}4^{-p\big [\frac{2 \kappa^\ast-1}{\kappa^\ast-1}\big ]^2}.
	\end{equation}
	To satisfy \eqref{eq_iter:bdd} we can choose $\nu^\ast$ small enough, i.e. we choose $\nu^\ast$ as
	\begin{equation*}
		\nu^\ast \leq C^{-\frac{2 \kappa^\ast-1}{\kappa^\ast-1}}4^{-p\big [\frac{2 \kappa^\ast-1}{\kappa^\ast-1}\big ]^2}.
	\end{equation*}
	Then
	\begin{equation} \label{lem3.8vfin}
		v(x,t) \geq 2^{- H - 1},
	\end{equation}
	for almost every $(x,t) \in B_{\infty} \times \Gamma_\infty^z$. Going back to $u$, \eqref{lem3.8vfin} implies
	\begin{equation} \label{lem3.8fin}
		u(x,\Lambda^{-1}(t)) \geq C M \frac{1}{\kappa \exp(\kappa^{p-1}\frac{1}{r^p}4 \tilde T)} \equiv
		\mu^\ast M,
	\end{equation}
	for almost every $t \in (2 \tilde T, 4 \tilde T)$.
	Define
	\begin{equation*}
		\hat C M^{2-p} r^p \equiv \Lambda^{-1}(4 \tilde T) = \frac{\exp(\kappa^{p-1}\frac{1}{r^p} C r^p)-1}{\kappa^{p-1}\frac{1}{r^p} M^{p-2}},
	\end{equation*}
	then we see that \eqref{lem3.8fin} implies the conclusion of the lemma. Moreover note that all constants are stable as $p \to 2$.
\end{proof}

\subsection{Proof of Lemma \ref{lem_expofpos}}
	Without loss of generality we may assume $t_0 = 0$ and, as in the statement of Lemma \ref{lem_expofpos},
	\begin{equation*}
		\big \{x \in B(x_0,r): u(x,0) > M \big \} \geq \delta |B(x_0,r)|.
	\end{equation*}
	Then, applying Lemma \ref{lem_preexpofpos} we first obtain that
	\begin{equation*}
		u(x,t) \geq \mu^\ast M,
	\end{equation*}
whenever $x \in B(x_0,2r)$ and for all Lebesgue instants $t$ for $u$ such that
\begin{equation*}
	t \in (\hat C M^{2-p} r^p/2,\hat C M^{2-p} r^p),
\end{equation*}
provided $\hat C M^{2-p}r^{p} < T_0$.
So in order to obtain the estimate from below in $B(x_0,r_0)$ we need to iterate Lemma \ref{lem_preexpofpos}, $\gamma = \log_2(r_0/r)$ times. Assume, without loss of generality, that $\gamma$ is an integer. Let
	\begin{align*}
		T^\ast_1 &\equiv
		\hat C M^{2-p}r^p + \hat C (\mu^\ast)^{2-p}2^p M^{2-p} r^p + \ldots +
		\hat C (\mu^\ast)^{(2-p)(\gamma-1)} 2^{p(\gamma-1)} M^{2-p} r^p.
	\end{align*}
	Using the definition of $\gamma$ we see that instead of $T^\ast_1$ we can take
	\begin{equation*}
		T_1 \equiv
		C \left ( \frac{r}{r_0} \right )^{\theta (2-p)}  M^{2-p} r_0^p.
	\end{equation*}
	Assume now that $T_1 < T_0$.
	 Hence if we at each step use Lemma \ref{lem_preexpofpos}, and choose a Lebesgue instant for $u$ in the allowed interval, we will end up with a
Lebesgue instant $t$ for $u$ satisfying
	\begin{equation*}
		T_1/2 < t < T_1.
	\end{equation*}
	At $t$ we have
	\begin{equation*}
		u(x,t) \geq (\mu^\ast)^{\gamma} M = \left ( \frac{r}{r_0} \right )^{\theta} M,
	\end{equation*}
	for all $(x,t) \in B(x_0,r_0) \times (T_1/2, T_1)$ and this completes the proof of Lemma \ref{lem_expofpos}.

\subsection{A reverse H\"older estimate for super-solutions}
\begin{lemma}\label{lem_5.3}
	Assume that $\A$ satisfies the structure conditions \eqref{admissiblesym}. Let $B(x_0,8r) \Subset \Omega$ and $0 < r < R$. Suppose that $u \geq 1$ is a weak super-solution to \eqref{eq_theeq} in an open set containing $\overline{B(x_0, 8r)} \times [0, 2^pr^p].$ Let $\gamma=1+(\kappa^\ast-1)/\kappa^\ast$
 and let $G$ be defined by the relation $\gamma= 1+1/G$. Given $ q\in (p-2 , p-2+ \gamma)$ and $s= p-2 + \gamma^{-l}(q-p+2)$, for some $l\in\{1,2,....\}$, there exists a
constant $C = C(\data ,\allowbreak q,\allowbreak s) \geq 1$ such that
\begin{align*}
	\Bigg( \frac{1}{2r^p} \int_0^{r^p} \fint_{B(0, \rho r)} &u^{q} d\mu dt\Bigg)^{1/(q-p+2)}\\
&\leq C \Bigg(
	 \frac{1}{ (2-\rho)^{pG+p} }
	 \frac{1}{r^p} \int_0^{2^pr^p} \fint_{B(0, 2r)} u^{s} d\mu dt\Bigg)^{1/(s-p+2)},	
\end{align*}
for all $1<\rho<2$.
\end{lemma}
\begin{proof} In this proof $C$ denotes a constant such that $C = C(\data,\allowbreak q, \allowbreak s)\geq 1$.
	We let \begin{eqnarray*}
\alpha_j&=&p-2 +{(q-p+2) \gamma^{j-l}},\\
\mathcal R_j &=& \Big(2- (2-\rho)\frac{1-2^{-j}}{1-2^{-l}}\Big)r,
\end{eqnarray*}
for $j=0,...,l$. Note that $\alpha_0=s$, $\alpha_l=q$, that $\{\alpha_j\}$ is increasing with $j$ and that $\{\mathcal R_j\}$ is decreasing with $j$. Hence, to prove the lemma it suffices to establish the estimate
\begin{equation*}
	\frac{1}{r^p} \int_0^
	{2r^p} \fint_{B(x_0, \mathcal R_{l})} u^{\alpha_l} d\mu dt\leq \Bigg(\frac{C}{(2-\rho)^{pG+p}}\frac{1}{r^p}\int_0^{2^pr^p} \fint_{B(x_0, 2r)} u^{\alpha_0} d\mu dt\Bigg)^{\gamma^l}.
\end{equation*}
 For each $j=0, \ldots,l$, set
\begin{equation*}
	U_j= B(x_0,\mathcal R_j)\times (0, \mathcal R^p_j) \text{ and } B_j=B(x_0,\mathcal R_j),
\end{equation*}
and let $\psi_j \in W_{\X,0}^{1,\infty}(B_j)$ be a test function as in Lemma \ref{lem_2.1} such that $\psi_j = 1$ in $B_{j+1}$ and $\psi_j = 0$ on $\partial B_{j}$. Let $\zeta_j \in C^{\infty}(0,\mathcal R_{j}^p)$ such that $\zeta_j(\mathcal R_j^p) = 0$, $\zeta_j = 1$ on $(0,\mathcal R_{j+1}]$ and
\begin{equation*}
	\Big|\frac{\partial \zeta_j}{\partial t}\Big| \leq C \frac{2^{pj}}{(2-\rho)^pr^p}.
\end{equation*}
Set $\varphi_j = \phi_j \zeta_j$, then
\begin{equation*}
		|\X \varphi_j| \leq C \frac{2^j}{(2-\rho) r},
	\end{equation*}
for $j=0,1, \ldots, l$.
Denote \begin{equation*}
\kappa_j =\frac{\alpha_{j+1}}{\alpha_j}, \qquad \beta_j = \frac{p\alpha_j}{\alpha_j -p+2}.
	\end{equation*}
Using \eqref{eq_D}, Lemma \ref{lem_parabolic:sobolev} and the fact that $1 < \rho < 2$ we obtain
\begin{eqnarray}\label{eq_5.4}
\fint_{U_{j+1}}u^{\alpha_{j+1}}d\mu dt &=& \fint_{U_{j+1}}u^{\kappa_j\alpha_j}d\mu dt\notag\\
& \leq&
C \fint_{U_{j}}\big(u^{\alpha_j/p} \varphi_j^{\beta_j/p}\big)^{\kappa_j p} d\mu dt \notag\\
&\leq& C \mathcal R_j^p\biggl ( \fint_{U_j} |\X (u^{\alpha_j/p} \varphi_j^{\beta_j/p})|^p d\mu dt\biggr )\notag\\
 &&\times\Big(\sup_{0<t<\mathcal R^p_j}\fint_{B_j} \big(u^{\alpha_j/p} \varphi_j^{\beta_j/p}\big)^{p (\kappa_j-1) G}d\mu \Big)^{1/G},
	\end{eqnarray}
	where $G = \frac{\kappa^\ast}{\kappa^\ast-1}$.
 Note that $\alpha_j(\kappa_j-1)G = \alpha_j -p +2$
 and $\beta_j(\kappa_j-1)G = p$. In view of this observation we can invoke Lemma \ref{cor_cacc}
 and apply it to the right hand side of \eqref{eq_5.4}, with\footnote{Note that $\xi\le 0$ for $j=0,...,l-1$}
 $\xi=\alpha_j -p +1$. This yields together with the fact that $u \geq 1$,
 \begin{align}\label{gga-}
	\sup_{0<t<\mathcal R^p_j}&\int_{B_j} \big(u^{\alpha_j/p} \varphi_j^{\beta_j/p}\big)^{p(\kappa_j-1)G} =
	\sup_{0<t<\mathcal R^p_j}\int_{B_j} u^{ \alpha_j -p +2} \varphi_j^p d\mu\notag \\
	&\leq C \bigg [ \int_{U_j}u^{\alpha_j} |\X \varphi_j|^p d\mu dt + \int_{U_j}u^{\alpha_j -p +2}\Big|\frac{\partial \varphi_j^p}{\partial t}\Big|d\mu dt\Big]\notag \\
	&\le C \Bigg(\frac{2^{-j}}{(2-\rho)r}\Bigg)^p \int_{U_j} u^{\alpha_j} d\mu dt.
\end{align}
Next, note that
\begin{eqnarray*}
	\int_{U_j}|\X (u^{\alpha_j/p} \varphi^{\beta_j/p}_j)|^p d\mu dt&\leq& \Big(\frac{\alpha_j}{p}\Big)^p \int_{U_j}|\X u|^p u^{\alpha_j - p}\varphi_j^p d\mu dt\notag\\
&& + \Big(\frac{\beta_j}{p}\Big)^p \int_{U_j}u^{\alpha_j}\varphi^{\beta_j}|\X \varphi_j|^p d\mu dt.
\end{eqnarray*}
Invoking once more Lemma \ref{cor_cacc} with $\xi = \alpha_j-p+1$ we obtain
\begin{align*}
 \int_{U_j}&|\X u|^p u^{\alpha_j - p}\varphi_j^p d\mu dt\notag
	\leq C\Biggl ( \int_{U_j}u^{\alpha_j}|\X \varphi_j|^p d\mu dt
	+\int_{U_j}u^{\alpha_j-p+2}\Big|\frac{\partial\varphi^p_j}{\partial t}\Big|d\mu dt\biggr ).
\end{align*}
Putting the together the estimates in the last two display we can conclude that
\begin{align}\label{gga}
	\int_{U_j}\big |\X \big (u^{\alpha_j/p} \varphi^{\beta_j/p}_j\big )\big |^p d\mu dt\leq C \Bigg(\frac{2^{-j}}{(2-\rho)r}\Bigg)^p \int_{U_j} u^{\alpha_j} d\mu dt.
\end{align}
Substituting the estimates in \eqref{gga-} and \eqref{gga} in \eqref{eq_5.4} yields
\begin{align*}
	\int_{U_{j+1}}u^{\alpha_{j+1}} d\mu dt
&\le C \Bigg[\mathcal R_j^p \Bigg(\frac{2^{-j}}{(2-\rho)r}\Bigg)^p \fint_{U_j} u^{\alpha_j}|\X \varphi_j|^p d\mu dt\Bigg]^\gamma \\
&\leq C \Bigg[\Bigg(\frac{2^{-j}}{(2-\rho)}\Bigg)^p \fint_{U_j} u^{\alpha_j}|\X \varphi_j|^p d\mu dt\Bigg]^\gamma.
\end{align*}
The proof now follows from a straightforward iteration argument as in \cite{K}. \end{proof}

\section{Proof of Theorem \ref{theorem1.2a}}
\noindent

\subsection{The Hot Alternative}

\begin{lemma}\label{hot} Assume that $\A$ satisfies the structure conditions \eqref{admissiblesym}. There exist constants $\sigma, \theta_h \in (0,1)$, and $T_h \geq 1$ all depending only on $\data$, such that the following holds. Let $B(x_0,8{\hat r}) \Subset \Omega$, $0 < 8{\hat r} < R$, and let $u$ be a super-solution to \eqref{eq_theeq} in an open set containing $\overline{B(x_0,8{\hat r})} \times [0, 2{\hat r}^p T_h]$. If for some $k > 8^{1/\sigma}$, and some Lebesgue instant $t_0$ for $u$ satisfying $0 < t_0 < 2^p {\hat r}^p$, one has
	\begin{equation*}
		|\{x \in B(x_0,2{\hat r}) : u(x,t_0) > 8 k^{1+\sigma} \}| > 8 k^{-\sigma} |B(x_0,2{\hat r})|,
	\end{equation*}
	then
	\begin{equation*}
		\inf_{B(x_0,2{\hat r}) \times ({\hat r}^p T_h, 2 {\hat r}^p T_h)} u \geq \theta_h.
	\end{equation*}
\end{lemma}

\begin{remark}
As remarked in \cite{DGVbook}, this result would follow from Lemma \ref{lem_expofpos}, if we could control the dependency of the constants with respect to the amount of positivity, i.e. the constant $\delta$ in Lemma \ref{lem_expofpos}. Instead we will employ Lemma \ref{lem_local_clust} to obtain a scale where we have the amount of positivity independent of the initial amount, with the scale instead depending on this initial amount in a power-like fashion. This allows us to iterate the expansion of positivity from this small initial datum to gain information on a large scale independent of $k$.
\end{remark}

\begin{proof} Let $\gamma = 8k^{-\sigma}$. Using Lemma \ref{lem_positivity} we obtain that
	\begin{equation}\label{eqksigma}
		|\{x \in B(x_0,2{\hat r}) : u(x,t) > 8k \}| > k^{-\sigma} |B(x_0,2{\hat r})|,
	\end{equation}
	holds for a.e. $t\in (t_0,t_0 + C k^{(1+\sigma)(2-p)} \gamma^{p/\hat \delta+1}{\hat r}^p)$. Let
	\begin{equation*}
		Q_1 = B(x_0,2{\hat r}) \times (t_0 + {T_1}/{2},t_0+T_1), \quad
		Q_2 = B(x_0,4{\hat r}) \times (t_0,t_0 + T_1),
	\end{equation*}
	where $T_1 =C k^{(1+\sigma)(2-p)} \gamma^{p/\hat \delta+1}{\hat r}^p$. Let $\psi \in W_{\X,0}^{1,\infty}(B(x_0,4r))$ and
$\zeta \in C^\infty(t_0,t_0 + T_1)$ be such that $\psi = 1$ on $B(x_0,2r)$, $\zeta=1$ on $(t_0 + T_1/2,t_0 + T_1)$, $\zeta(t_0) = 0$, and
	\begin{equation*}
		0 \leq \zeta_t \leq \frac{C} {T_1} \quad \text{and} \quad |\X \psi| \leq \frac{C}{{\hat r}}.
	\end{equation*}
	Then, using Remark \ref{lem3.2rem} (with $(u-k^{1+\sigma})_-$, $\epsilon = 1$, $\phi = \psi \zeta$) and \eqref{eq_D}, we see that
	\begin{eqnarray}\label{eqgradbdd}
		\int\limits_{Q_1 \cap [u < k^{1+\sigma}]} |\X u|^p d\mu dt&\leq& C \int\limits_{Q_2} \bigg ( {\hat r}^{-p}  (u-k^{1+\sigma})_-^{p} +\frac{1}{T_1} (u-k^{1+\sigma})_-^2 \bigg ) d\mu dt \notag \\
		&\leq& C \bigg ({\hat r}^{-p} k^{(1+\sigma)p} + \frac{1} {T_1} k^{(1+\sigma)2} \bigg ) |Q_2| \notag \\
		&\leq& C \frac{k^{(1+\sigma)p} }{{\hat r}^p \gamma^{p/\hat \delta+1}}|Q_1|.
	\end{eqnarray}
	Let $w = (k^{1+\sigma}-u)_+/k^{1+\sigma}$ and $z = 8(1-w)/\gamma$. Then, using \eqref{eqksigma} we see that
	\begin{equation*}
		\{[z > 1] \cap B(x_0,2{\hat r})\} = \{[u > 8k] \cap B(x_0,2{\hat r})\} > \frac{\gamma}{8} |B(x_0,2{\hat r})|.
	\end{equation*}
	Rewriting \eqref{eqgradbdd} for $z$ we obtain the estimate
	\begin{equation*}
		\int_{Q_1} |\X z|^p d\mu dt 
		\leq \frac{\tilde C}{\gamma^{p/\hat \delta+p+1}{\hat r}^p} |Q_1|.
	\end{equation*}
	Hence, for some Lebesgue instant $\tau_1$ for $z$ satisfying $\tau_1 \in (t_0+{T_1}/{2},t_0+T_1)$ we have
	\begin{equation*}
		\left ( \fint_{B(x_0,2{\hat r})} |\X z(\cdot, \tau_1)|^p d\mu \right )^{1/p} \leq \frac{C}{\gamma^{(p/\hat \delta+1)/p+1} {\hat r}}.
	\end{equation*}
	Using Lemma \ref{lem_local_clust} with $\hat \gamma = \frac{C}{ \gamma^{(p/\hat \delta+1)/p+1}}$, $\lambda = \delta = 1/2$, and $\alpha = \gamma/8$, we
obtain 	\begin{equation} \label{eqzhalf}
		\{[z(\cdot, \tau_1) > 1/2] \cap B(y_0,2\epsilon {\hat r}) \} > \frac{1}{2} |B(y_0,2\epsilon {\hat r})|,
	\end{equation}
	for some $y_0 \in B(x_0,2{\hat r})$, with
	\begin{equation*}
		\epsilon = \frac{1}{C_1} \gamma^{\frac{\kappa^\ast+1}{\kappa^\ast p}+\frac{p+\hat \delta}{\hat \delta p}+1},
	\end{equation*}
	where $C_1$ is to be chosen.
 Going back to $u$, \eqref{eqzhalf} becomes
	\begin{equation*}
		\{[u(\cdot, \tau_1) > 4k] \cap B(y_0,2 \epsilon {\hat r}) \} > \frac{1}{2} |B(y_0,2 \epsilon {\hat r})|.
	\end{equation*}
	Next we use Lemma \ref{lem_expofpos} with $\delta = 1/2$, $r = 2\epsilon {\hat r}$, $r_0 = 4{\hat r}$ $t_0 = \tau_1$ and $M = k$, to conclude that
	\begin{equation*}
		\inf_{Q'} u \geq k \left ( 4 \epsilon \right )^\theta,
	\end{equation*}
	where $Q' = B(y_0,4{\hat r}) \times (\tau_1 + T/2, \tau_1 + T)$, with $T = \tilde \gamma \left ( k \left ( 2 \epsilon \right )^\theta \right )^{2-p} {\hat r}^p$. Now let
	\begin{equation*}
		\sigma = \frac{1}{\big (\frac{\kappa^\ast+1}{\kappa^\ast p}+\frac{p+\hat \delta}{\hat \delta p}+1 \big )\theta}.
	\end{equation*}
	Then, since $\gamma = 8k^{-\sigma}$ we see that we can take $C_1 = C_1(\data) \geq 1$ large enough such that
	\begin{equation*}
		\frac{1}{C} \leq k(2\epsilon)^\theta \leq 1
	\end{equation*}
	independent of $k$. This completes the proof of the lemma.
\end{proof}

\subsection{The Cold Alternative}

\begin{lemma}\label{lem_5.1} Assume that $\A $ satisfies the structure conditions \eqref{admissiblesym}. Let $\sigma = \sigma(\data)$ be as in Lemma \ref{hot}. There exist positive
constants $T_c$, $M_c \geq 1$ and $\theta_c \in (0,1)$, all depending only on $\data$, such that the following holds. Let $B(x_0,8 \hat r) \Subset \Omega$, $0 < 8 \hat r < R$ , and let $u$ be a weak super-solution to \eqref{eq_theeq} in
an open set containing $\overline{B(x_0, 8{\hat r})} \times [0, 2^p{\hat r}^pT_c]$. Assume that $t= 0$ is a Lebesgue instant for $u$. Assume that,
\begin{equation}\label{6.0}
	\fint_{B(x_0,{\hat r})} u(x,0) d\mu \geq M_c,
\end{equation}
and that
\begin{equation}\label{eq_5.2}
	|\{ x\in B(x_0,2{\hat r}): u(x,t)> 8 k^{1+\sigma}\}|\leq 8 k^{-\sigma}|B(x_0,2{\hat r})|,
\end{equation}
for every $k\geq 8^{1/\sigma}$ and for almost every $0 < t < 2^p {\hat r}^p $. Then
\begin{equation*}
	\inf_{B(x_0,2{\hat r})\times ( {\hat r}^pT_c, 2 {\hat r}^pT_c)}u\geq \theta_c.
\end{equation*}
\end{lemma}

\begin{lemma} \label{lem_5.5} Assume that $\A$ satisfies the structure conditions \eqref{admissiblesym} and let $u$ be a weak super-solution to \eqref{eq_theeq} in an open set compactly containing $B(x_0, 2{\hat r}) \times (0, 2^p{\hat r}^p)$, $B(x_0,2\hat r) \Subset \Omega$, with $t=0$ a Lebesgue instant. Set $\gamma=1+({\kappa^\ast-1})/{\kappa^\ast}$. If $u$ satisfies \eqref{eq_5.2} then for all $q\in (p - 2 ,p - 2+\gamma)$ there exists a constant $C = C(\data,\allowbreak q) \geq 1$ such that
\begin{equation}\label{conclusion-5.2}
	\fint^{2^p{\hat r}^p}_0\fint_{B(x_0,3/2{\hat r})}u^q d\mu dt \leq C.
\end{equation}
\end{lemma}
\begin{proof} Consider the super-solution $v = u + 1$ and set
$\delta = \frac{\sigma}{2 (1+ \sigma)}.$ where $\sigma > 0$ is as in Lemma \ref{hot}. The hypothesis \eqref{eq_5.2} yields that there exists $C= C(\delta) \geq 1$, such that
\begin{equation}\label{eq_5.6}
	\int_{B(x_0,2{\hat r})}v^{\delta}(x, t) d\mu \leq C\Big ( 1+\sum_{j=1}^\infty \int_{8 k^{j(1+\sigma)}<u< 8 k^{(j+1)(1+\sigma)}}u^{\delta}(x, t) d\mu\Big )
\end{equation}
\begin{equation*}
	\le C|B(x_0,2{\hat r})| \Big(1+ \sum_{j=0}^\infty k^{-\sigma j/2} \Big) \leq C |B(x_0,2{\hat r})|,
\end{equation*}
for $k$ sufficiently large and for almost every $0 < t < 2^p {\hat r}^p$.
 In view of Lemma \ref{lem_5.3} and H\"older's inequality it suffices to prove that \eqref{conclusion-5.2} holds for
\begin{equation*}
	q=p-2+\gamma\delta.
\end{equation*}
Let $U(s) = B(x_0, s{\hat r}) \times (0, 2^p{\hat r}^p)$,
and consider the function from Lemma \ref{lem_2.1}, $\varphi \in W_{\X,0}^{1,\infty} (B(x_0, S{\hat r}))$, $0 \leq \varphi \leq 1$,
$\varphi = 1$ in $B(x_0,s\hat r)$ and $|\X \varphi| \leq C{\hat r}^{-1}/(S - s)$,
where $7/4 \leq s < S \leq 2$. Set\footnote{Note that $\alpha (\kappa-1)G=\delta$}
\begin{equation*}
\alpha = p - 2 + \delta, \kappa = \frac{p - 2 + \gamma \delta}{p - 2 + \delta}
.
\end{equation*}
Using Lemma \ref{lem_parabolic:sobolev} (see also \eqref{eq_5.4}), and together with \eqref{eq_5.6}, we see that the following holds
\begin{align}
	\int_{U(s)} v^q d\mu dt &\le
	\int_{U(S)}v^{q} \varphi^{\kappa p} d\mu dt
 	=\int_{U(S)} (u^{\alpha/p} \varphi)^{\kappa p} d\mu dt \notag \\
	&\leq C{\hat r}^p \int_{U(S)} |\X (v^{\alpha/p} \varphi)|^p d\mu dt \Big( \sup_{0<t<2^p {\hat r}^p}\fint_{B(x_0,2{\hat r})} \big(v^{\alpha/p} \varphi \big)^{p(\kappa-1)G} \Big)^{1/G} \notag \\
&\leq C{\hat r}^p \int_{U(S)} |\X (v^{\alpha/p} \varphi)|^p d\mu dt \Big( \sup_{0<t<2^p {\hat r}^p}\fint_{B(x_0,2{\hat r})} v^{\delta} d\mu \Big)^{1/G} \notag \\
&\leq C {\hat r}^{p} \int_{U(S)} |\X (v^{\alpha/p} \varphi)|^p d\mu dt, \label{last-1}
\end{align}
where $C\geq 1$ also depends on $\delta$. Using Lemma \ref{lem_cacc} with $\xi = -1 + \delta$, H\"older's inequality and Young's inequality we obtain
\begin{eqnarray}
	\int_{U(S)}|\X (v^{\alpha/p} \varphi)|^p d\mu dt &\leq& C \int_{U(S)} v^\alpha |\X \varphi|^p d\mu dt+C
	\int_{U(S)} v^{(\frac{\alpha}{p}-1)p}|\X v|^p \varphi^p d\mu dt \notag\\
	&\leq& C\int_{U(S)} v^\alpha |\X \varphi|^p d\mu dt + C\sup_{0<t<2^p{\hat r}^p}\int_{B(X_0,S{\hat r})}v^\delta d\mu \notag \\
	&\leq& \frac{1}{2 C \hat r^p} \int_{U(S)} v^q d\mu dt + C \hat r^{\frac{p\alpha}{q-\alpha}} \int_{U(S)} |\X \varphi|^{\frac{pq}{q-\alpha}} d\mu dt\notag\\
 &&+ C\sup_{0<t<2^p{\hat r}^p}\int_{B(X_0,S{\hat r})}v^\delta d\mu , \label{last-2}
\end{eqnarray}
	again $C \geq 1$ depends on $\delta$.	Using \eqref{eq_5.6}, \eqref{last-1}, \eqref{last-2} we obtain
\begin{equation*}
	\int_{U(s)} v^q d\mu dt \le \frac{1}{2} \int_{U(S)} v^q d\mu dt
	+ C \frac{ |B(x_0,2{\hat r})| {\hat r}^p }{(S-s)^{\frac{p q}{q-\alpha}}}+ C |B(x_0,2{\hat r})|{\hat r}^p.
\end{equation*}
Now, using the ``iteration lemma'' in \cite[Lemma 5.1]{Giaquinta} we deduce that
\begin{equation*}
	\int_{U(7/4)} v^q d\mu dt \le C |B(x_0,2{\hat r})|{\hat r}^p.
\end{equation*}
This proves \eqref{conclusion-5.2} for $q = p-2+\gamma \delta$, with a constant $C$ which also depends on $\delta$.
\end{proof}

\begin{lemma}\label{lem_5.7} In the same hypothesis as Lemma \ref{lem_5.5}, there exists a constant
$C \geq 1$ depending only on $\data$ such that
\begin{equation*}
	\fint^{{\hat r}^p}_0\fint_{B(x_0,\frac{5{\hat r}}{4})}|\X u|^{p-1} d\mu dt \leq C{\hat r}^{1-p}.
\end{equation*}
\end{lemma}
\begin{proof} Set $\xi=-\frac{\kappa^\ast-1}{2\kappa^\ast (p-1)}$, $q= (p-1)(1-\xi)$ and $v = u + 1$.
Choose a test function as in Lemma \ref{lem_2.1}, $\phi \in W_{\X,0}^{1,\infty}(B(x_0, 3/2{\hat r}))$ such that $\phi = 1$ in $B(x_0,5/4\hat r)$. Next choose $\zeta \in C^\infty (0, 2^p{\hat r}^p)$, $0 \leq \zeta \leq 1$ such that $\zeta = 1$ on $(0,\hat r^p)$, $\zeta(2^p \hat r^p) = 0$ and such that for $\varphi = \phi \zeta$
\begin{equation} \label{eq_testbdd}
	|\X \varphi| + \bigg(-\frac{\partial\varphi}{\partial t}\bigg)_+\leq \frac{C}{\hat r^p}.
\end{equation}
Using H\"older's inequality and Lemma \ref{lem_5.5} we first see that
\begin{align} \label{eq_gradbdd}
	\fint^{{\hat r}^p}_0 &\fint_{B(x_0,5/4{\hat r})}|\X u|^{p-1} d\mu dt = \fint^{{\hat r}^p}_0\fint_{B(x_0,5/4{\hat r})}|\X v|^{p-1} v^{-q/p} v^{q/p} d\mu dt \notag \\
	&\leq
\bigg( \fint^{{\hat r}^p}_0\fint_{B(x_0,5/4{\hat r})}|\X v|^p v^{-1+\xi} d\mu dt\bigg)^{(p-1)/p}
\bigg( \fint^{{\hat r}^p}_0\fint_{B(x_0,5/4{\hat r})}v^q d\mu dt\bigg)^{1/p} \notag \\
&\le C \bigg( \fint^{2^p{\hat r}^p}_0\fint_{B(x_0,3/2{\hat r})}|\X v|^p v^{-1+\xi} \varphi^p d\mu dt\bigg)^{(p-1)/p}.
\end{align}
Furthermore, using \eqref{eq_testbdd}, Lemma \ref{lem_cacc} and H\"older's inequality we conclude that
\begin{align} \label{eq_gradbdd2}
\fint^{2^p{\hat r}^p}_0 \fint_{B(x_0,3/2{\hat r})}|\X v|^p v^{-1+\xi} &\varphi^p d\mu dt \notag \\
\leq &C \fint^{2^p{\hat r}^p}_0\fint_{B(x_0,3/2{\hat r})} v^{p-1+\xi}|\X \varphi|^p d\mu dt\notag \\
&+ C\fint^{2^p{\hat r}^p}_0\fint_{B(x_0,3/2{\hat r})} v^{1+\xi}\bigg(
-\frac{\partial\varphi^p}{\partial t}\bigg)_+
d\mu dt \notag \\
\leq &C{\hat r}^{-p}.
\end{align}
Putting together \eqref{eq_gradbdd} and \eqref{eq_gradbdd2} we see that the proof of Lemma \ref{lem_5.7} is complete.
\end{proof}

\begin{lemma} \label{lem_5.8} In the same hypothesis as Lemma \ref{lem_5.5}, there exists a constant $M_c \geq 1$, depending only on $\data$, such that if \eqref{6.0} is satisfied then
\begin{equation*}
	\inf_{0<t<{\hat r}^p}\fint_{B(x_0,5/4{\hat r})}u(x, t) d\mu \geq \frac{M_c}{2}.
\end{equation*}
\end{lemma}
\begin{proof} Consider a test function as in Lemma \ref{lem_2.1}, $\varphi \in W_{\X,0}^{1,\infty} (B(x_0, 5/4{\hat r}))$, with
$0 \le \varphi \le 1$, $\varphi = 1$ in $B(x_0, {\hat r})$ and
$|\X \varphi | \leq C/{\hat r}$. Using the hypothesis that the time $t=0$ is a Lebesgue instant and applying a standard approximation argument (see for instance \cite[Remark 2.4]{K}) to the definition of super-solution one easily obtains that for all Lebesgue instants $t\in (0,{\hat r}^p)$
\begin{align}
	\fint_{B(x_0,5/4{\hat r})}&u(x, t)\varphi(x) d\mu - \fint_{B(x_0,5/4{\hat r})}u(x, 0)\varphi(x) d\mu\notag \\
	\geq& \int_0^t \fint_{B(x_0,5/4{\hat r})}\A(x,t,u,Xu) \cdot X\varphi d\mu dt\notag \\
	\geq& - \A_1 \int_0^t \fint_{B(x_0,5/4{\hat r})} |Xu|^{p-1}|X\varphi|d\mu dt \notag \\
	\geq& -C{\hat r}^{p-1} \fint_0^t \fint_{B(x_0,5/4{\hat r})} |Xu|^{p-1}d\mu dt \ge -C, \label{last-3}
\end{align}
where in the last line we have used Lemma \ref{lem_5.7}.
From \eqref{last-3} it follows that if $M_c$ is taken to satisfy $M_c \geq \frac{2}{3} C$, we obtain the lemma.
\end{proof}

\noindent{\it Proof of Lemma \ref{lem_5.1}}. In view of \eqref{eq_5.2} and Lemma \ref{lem_5.5} we have
\begin{equation} \label{bdd++}
	\fint_0^{2^p{\hat r}^p} \fint_{B(x_0,3/2 {\hat r})} u^q d\mu dt \le C,
\end{equation}
for all $q\in (p-2,p-2+\gamma)$ and $\gamma=1+\frac{\kappa^\ast-1}{\kappa^\ast}$. From \eqref{bdd++}, Lemma \ref{lem_5.8} and H\"older's inequality we see
 that,
\begin{eqnarray}
	\frac{M_c}{2} &\leq& \fint_0^{2^p{\hat r}^p} \fint_{B(x_0, 5/4 {\hat r})} u(x,t) d\mu dt\notag\\
 &\leq& \frac{1}{|Q|} \int_{\{(x,t)\in Q | u\ge M_c/4\} } \!\!\!\!\!\! u(x,t) d\mu dt + \frac{M_c}{4} \notag\\
&\leq& C \Bigg( \frac{| \{(x,t)\in Q | u\ge M_c/4\}|}{ 2^p{\hat r}^p|B(x_0,5/4{\hat r})|}\Bigg)^{\frac{q-1}{q}} + \frac{M_c}{4}, \label{last-4}
\end{eqnarray}
where $Q=B(x_0,5/4{\hat r})\times (0,2^p {\hat r}^p)$. Using \eqref{last-4} we can conclude that there must exists a Lebesgue instant $t_0$ for $u$ satisfying $t_0\in (0,2^p {\hat r}^p)$ and
\begin{equation*}
	\Big\lvert \Big\{ x\in B\Big(x_0,\frac{5}{4} {\hat r}\Big) \Big | u(x,t_0)\ge \frac{M_c}{4}\Big\}\Big \rvert \ge \frac{1}{C} \Big|B\Big(x_0,\frac{5}{4} {\hat r}\Big) \Big|.
\end{equation*}
Lemma \ref{lem_5.1} now follows by choosing $T_c$ sufficiently large so that one can apply the expansion of positivity, Lemma \ref{lem_expofpos}.

\subsection{Final argument}
Let $t_0 < t_1 < t_0+T_0$ be a Lebesgue instant of $u$. Consider the constants $ T_h,T_c, M_c$ from Lemma \ref{hot} and Lemma \ref{lem_5.1}, set $\bar T=\max\{T_h,T_c\}$ and let
\begin{equation*}
	N =\fint_{B(x_0,r)}u(x, t_1) d\mu.
\end{equation*}
Assume that
\begin{equation} \label{Nbig}
	N \ge \Big(\frac{2M_c^{p-2} \bar T r^p}{T_0+t_0-t_1}\Big)^{\frac{1}{p-2}},
\end{equation}
since otherwise there is nothing to prove. Next, consider the time-rescaled function
\begin{equation*}
	v(x, \tau) =\frac{M_c}{N} u(x, t_1 + (M_c/N)^{p-2} \tau).
\end{equation*}
	Note that assumption \eqref{Nbig} implies
	\begin{equation*}
		 2\bar Tr^p \le \Big(\frac{N}{M_c}\Big)^{p-2} (T_0+t_0-t_1),
	\end{equation*}
 and that $v$ is a super-solution of an equation whose symbol $\tilde \A$ which is structurally similar to $\A$, in the cylinder
\begin{equation*}
	B(x_0, 8r)\times(0, 2\bar T r^p) \subset B(x_0, 8r)\times\bigg(0, \bigg(\frac{N}{M_c}\bigg)^{p-2}(T_0 + t_0 - t_1)\bigg),
\end{equation*}
with
\begin{equation*}
	\fint_{B(x_0,r)}v(x, 0) d\mu = M_c.
\end{equation*}
At this point we can invoke the hot and cold alternatives. Indeed, applying Lemma \ref{hot} and Lemma \ref{lem_5.1} to the super-solution $v$, in the cylinder $B(r, 8r)\times(0, 2\bar T r^p)$, we obtain
\begin{equation*}
	\inf_{B(x_0,2r)\times(T_hr^p,2T_hr^p)}v \ge \theta_h \text{ or } \inf_{B(x_0,2r)\times(T_cr^p,2T_cr^p)}v \ge \theta_c.
\end{equation*}
Applying Lemma \ref{lem_expofpos} once more gives
\begin{equation*}
	\inf_{B(x_0,2r)\times(\bar T r^p,2\bar T r^p)}v \ge \min\{\theta_h, \theta_c\}.
\end{equation*}
Returning to the original variables yields the conclusion of Theorem \ref{theorem1.2a}.

\section{Proof of Theorem \ref{theorem1.2b}}
\noindent
We begin by establishing local boundedness of sub-solutions.

\begin{lemma} \label{lem_Mos1}
	Assume that $\A$ satisfies the structure conditions \eqref{admissiblesym}. Let $B(x_0,{\hat r}) \Subset \Omega$, $0 < {\hat r} < R$, let
$u$ be a sub-solution in an open set containing the closure of $Q=B(x_0,{\hat r}) \times (t_0 - T_0, t_0)$, and assume that almost everywhere in $Q$,
	\begin{equation}\label{assump}
		u \geq \bigg ( \frac{{\hat r}^p}{T_0} \bigg )^{\frac{1}{p-2}} > 0.
	\end{equation}
If $\delta_0 > 0$ and $G = \frac{\kappa^\ast}{\kappa^\ast - 1}$, then there exists a constant $C = C(\data) \geq 1$ such that
	\begin{equation*}
		\sup_{B(x_0,\sigma {\hat r})\times (t_0-\sigma^p T_0,t_0)} u \leq \bigg ( \frac{T_0}{{\hat r}^p} \frac{C}{(1-\sigma)^{pG+p}} \fint_Q u^{p-2+\delta} d\mu dt\bigg )^{1/\delta},
	\end{equation*}
	for every $\delta \geq \delta_0$ and $0 < \sigma < 1$.
\end{lemma}
\begin{proof}
	This proof is similar to \cite[Lemma 4.6]{K1}, but we point out the relevant changes.
	Let $\sigma {\hat r} \leq s < S < {\hat r}$. We set
	\begin{equation*}
		r_0 = S, \quad r_j = S-(S-s)(1-2^{-j}), \quad j=1,2,\ldots
	\end{equation*}
	and
	\begin{eqnarray*}
		U_j &=& B_j \times \Gamma_j = B(x_0,r_j) \times (t_0-(r_j/{\hat r})^p T_0, t_0), \\
		U(S) &=& B(x_0,S) \times (t_0 - (S/{\hat r})^p T_0, t_0).
	\end{eqnarray*}
	We choose test functions $\psi_{j} \in W_{\X,0}^{1,\infty}(B_j)$, and $\zeta_j \in C^{\infty}(\Gamma_j)$ such that for $\phi_j = \psi_j \zeta_j$ we have
	\begin{equation*}
		0 \leq \phi_j \leq 1, \quad \phi_j=1 \text{ in } U_{j+1}, \phi_j = 0 \text{ on } \partial_p U_j,
	\end{equation*}
	and
	\begin{equation} \label{eqphi}
		|\X \phi_j| \leq \frac{C}{S-s}2^j, \quad \Big \lvert \frac{\partial \phi_j}{\partial t}\Big \rvert \leq \frac{{\hat r}^p}{T_0} \frac{C}{(S-s)^p}2^{pj}.
	\end{equation}
	From Lemma \ref{lem_parabolic:sobolev} we obtain
	\begin{align}
		\fint_{U_j} &\big (u^{\alpha/p} \phi_j^{\beta/p}\big )^{\kappa p} d\mu dt \notag \\
		&\leq C r_j^p \fint_{U_j} \big |\X \big (u^{\alpha/p} \phi_j^{\beta/p}\big )\big |^p d\mu dt \bigg ( \sup_{\Gamma_j} \fint_{B_j} \big (u^{\alpha/p} \phi_j^{\beta/p}\big )^{p(\kappa-1)G} d\mu \bigg )^{\frac{1}{G}} \notag \\
		&\equiv C r_j^p \frac{I_1}{|U_j|}\left ( \frac{I_2}{|B_j|} \right )^{\frac{1}{G}}, \label{eqI1I2}
	\end{align}
	where
	\begin{equation*}
		\alpha = p-1+\xi,\quad \kappa = 1 + \frac{(1+\xi)}{G\alpha},\quad \beta = \frac{p \alpha}{1+\xi},
	\end{equation*}
	for $\xi \geq 1$. First, using Lemma \ref{lem_cacc} we see that
	\begin{align}
		I_2 = \sup_{\Gamma_j} \int_{B_j} u^{1+\xi} \phi_j^p d\mu
		\leq C \frac{1+\xi}{\xi^{p-1}}\int_{U_j} u^{p-1+\xi}|\X \phi_j|^p + u^{1+\xi} \Big \lvert \frac{\partial \phi_j}{\partial t}\Big \rvert \phi_j^{p-1} d\mu dt. \label{eqI2}
	\end{align}
	Next, using that $\phi_j^{\beta} \leq \phi_j^p$, Young's inequality and Lemma \ref{lem_cacc} we deduce that
	\begin{align}
		I_1 &\leq
		C \left ( \alpha^p \int_{U_j} |\X u|^p u^{-1+\xi} \phi_j^{p} d\mu dt+ \int_{U_j} |\X \phi_j|^p u^{p-1+\xi} d\mu dt \right ) \notag \\
		&\leq C \left ( 1 + \frac{\alpha^p}{\xi^p} \right )\int_{U_j} |\X \phi_j|^p u^{p-1+\xi} d\mu dt\notag\\
 &+ C \frac{\alpha^p}{\xi (1+\xi)} \int_{U_j} u^{1+\xi} \Big \lvert \frac{\partial \phi_j}{\partial t} \Big \rvert \phi_j^{p-1} d\mu dt. \label{eqI1}
	\end{align}
Using the assumption in \eqref{assump} we see that almost everywhere in $Q$
	\begin{equation} \label{u1eps}
		u^{1+\xi} \leq \frac{T_0}{{\hat r}^p} u^{p-1+\xi}.
	\end{equation}
	Furthermore, using \eqref{u1eps}, \eqref{eqI1}, \eqref{eqI2} and \eqref{eqphi} we can conclude that
	\begin{align}
		\fint_{U_j} &\big |\X \big (u^{\alpha/p} \phi_j^{\beta/p}\big )\big |^p d\mu dt \bigg ( \sup_{\Gamma_j} \frac{1}{T_j}\fint_{B_j} \big (u^{\alpha/p} \phi_j^{\beta/p}\big )^{p(\kappa-1)G} d\mu \bigg )^{\frac{1}{G}} \notag \\
		&\leq C\Bigg ( \xi^p \fint_{U_j} \bigg (|\X \phi_j|^p u^{p-1+\xi} + u^{1+\xi} \bigg |\frac{\partial \phi_j}{\partial t}\bigg | \bigg ) d\mu dt \Bigg )^\gamma \notag \\
		&\leq C \bigg ( \frac{ 2^{jp} \xi^p}{(S-s)^p} \fint_{U_j} u^{p-1+\xi} d\mu dt \bigg )^\gamma, \label{eqbdd+++}
	\end{align}
	where $\gamma=1+\frac{\kappa^\ast-1}{\kappa^\ast}$. From \eqref{eqI1I2}, \eqref{eqbdd+++} and the definition of $T_j$, we also obtain that
	\begin{eqnarray}
		\int_{U_{j+1}} u^{\kappa \alpha} d\mu dt &\leq&
		|U_j| \fint_{U_j} u^{\kappa \alpha} d\mu dt \notag \\
		&\leq& C r_j^p |U_j| T_j^{\frac{1}{G}}
		 \bigg ( \frac{ 2^{jp} \xi^p}{(S-s)^p} \fint_{U_j} u^{p-1+\xi} d\mu dt \bigg )^\gamma \notag \\
		&=& C r_j^{p\gamma} |U_j| \bigg ( \frac{T_0}{{\hat r}^p} \bigg )^{\frac{1}{G}}\bigg ( \frac{ 2^{jp} \xi^p}{(S-s)^p} \fint_{U_j} u^{p-1+\xi} d\mu dt \bigg )^\gamma. \label{eqkalpha}
	\end{eqnarray}
	Next, using \eqref{eq_D} we see that we can rewrite \eqref{eqkalpha} as
	\begin{align*}
		\fint_{U_{j+1}} u^{\kappa \alpha} d\mu dt &\leq C r_j^{p\gamma} \bigg ( \frac{r_j}{r_{j+1}} \bigg )^{N+p} \bigg ( \frac{T_0}{{\hat r}^p} \bigg )^{\frac{1}{G}}\bigg ( \frac{ 2^{jp} \xi^p}{(S-s)^p} \fint_{U_j} u^{p-1+\xi} d\mu dt \bigg )^\gamma \\
		&\leq C \Bigg ( \frac{r_j^{N+p}}{r_{j+1}^{(N+p)/\gamma}} \bigg ( \frac{T_0}{{\hat r}^p} \bigg )^{\frac{\gamma-1}{\gamma}} \frac{ 2^{jp} \xi^p}{(S-s)^p} \fint_{U_j} u^{p-1+\xi} d\mu dt \Bigg )^\gamma.
	\end{align*}
	 Now the $r_j$,$r_{j+1}$, $T_0/{\hat r}^p$ dependence is exactly as in \cite[(4.8)]{K1} with the difference that now $\gamma=1+\frac{\kappa^\ast-1}{\kappa^\ast}$. This implies that we can proceed with the iteration procedure as in \cite[Lemma 4.6]{K1} obtain the conclusion of the lemma.
\end{proof}

\begin{lemma} \label{lem_weak:Harnack:sub}
	Let $u$ be as in Lemma \ref{lem_Mos1}. There exists a constant $C = C(\data) \geq 1$ such that
	\begin{equation*}
		\sup_{B(x_0,\sigma {\hat r})\times (t_0-\sigma^p T_0,t_0)} u \leq \frac{T_0}{{\hat r}^p} \frac{C}{(1-\sigma)^{p+p^2 G}} \left ( \sup_{t_0-T_0 < t < t_0} \fint_{B(x_0,{\hat r})} u\ d\mu\right )^{p-1},
	\end{equation*}
	with $G = \kappa^\ast/(\kappa^\ast - 1)$,
	for every $0 < \sigma < 1$.
\end{lemma}
\begin{proof}
	By using \eqref{eq_D} and Lemma \ref{lem_Mos1}, we can argue as in \cite[Lemma 4.9]{K1} to obtain the conclusion of the lemma.
\end{proof}
Theorem \ref{theorem1.2b} now follows as a simple consequence of Lemma \ref{lem_weak:Harnack:sub}.
\def\cprime{$'$}

\end{document}